\documentclass[11pt, letterpaper]{article}

\usepackage[margin=1 in]{geometry}
\usepackage{mathtools}
\usepackage{amssymb} 
\usepackage{enumitem}
\usepackage{amsthm}
\usepackage{amsmath}
\usepackage{amsopn}
\usepackage{bbm}
\usepackage[numbers]{natbib}
\usepackage{enumitem}
\DeclarePairedDelimiterX\Basics[1](){\let\given\sgiven #1}
\usepackage[utf8]{inputenc}
\usepackage{color}
\usepackage{caption}
\usepackage{subcaption}
\usepackage{enumitem}
\usepackage{pifont}
\usepackage{algorithm}
\usepackage[noend]{algpseudocode}

\usepackage[T1]{fontenc}
\usepackage{etoolbox}

\usepackage{float}
\newtheorem{theorem}{Theorem}

\newtheorem{corollary}{Corollary}[theorem]
\newtheorem{remark}{Remark}[theorem]
\newtheorem{lemma}[theorem]{Lemma}
\newtheorem{definition}[theorem]{Definition}
\newtheorem{proposition}[theorem]{Proposition}
\numberwithin{equation}{section}
\newtheorem*{assumption*}{\assumptionnumber}
\newcommand{\innermid}{\nonscript\;\delimsize\vert\nonscript\;}
\newcommand{\activatebar}{%
  \begingroup\lccode`\~=`\|
  \lowercase{\endgroup\let~}\innermid 
  \mathcode`|=\string"8000
}
\providecommand{\assumptionnumber}{}
\makeatletter

\newlist{steps}{enumerate}{1}
\setlist[steps, 1]{label = Step \arabic*:}

\let\given\givenbase

\newcommand{\interior}[1]{%
  {\kern0pt#1}^{\mathrm{o}}%
}

\begin{document}
\title{Stochastic Games for Fuel Followers Problem: $N$ versus MFG}
\author{
Xin Guo
\thanks{Department of Industrial Engineering and Operations Research, University of California, Berkeley, USA. Email: xinguo@berkeley.edu.}
\and Renyuan Xu
\thanks{Department of Industrial Engineering and Operations Research, University of California, Berkeley, USA. Email: renyuanxu@berkeley.edu.}
}
\date{\today}
\maketitle

\begin{abstract}
In this paper we formulate and analyze an $N$-player stochastic game of the classical fuel follower problem and  its Mean Field Game (MFG) counterpart.
For the $N$-player game, we obtain the Nash Equilibrium (NE)  explicitly by deriving and analyzing a system of Hamilton--Jacobi--Bellman (HJB) equations,  and by establishing the existence of a unique strong solution to the associated Skorokhod problem on an unbounded polyhedron with an oblique reflection. For the MFG, we derive a
bang-bang type NE under some mild technical conditions and by the viscosity solution approach. We also  show that this solution is an $\epsilon$-NE to the $N$-player game, with $\epsilon =O(\frac{1}{\sqrt{N}})$. 
The $N$-player game and the MFG differ in that the NE for the former is state dependent while the NE for the latter is {threshold-type bang-bang policy where the threshold is state independent
}. Our analysis shows that the NE for a
stationary MFG may not be the NE for the corresponding MFG.

\end{abstract}

\section{Introduction}

The classic fuel follower problem  concerns controlling a single moving object on a real line whose movement is modeled by a standard Brownian motion. The controller controls the position of her object in a possibly non-continuous way, i.e., with singular controls.
Her objective is to minimize over an infinite-time horizon, the total amount of control and the total $L^2$ distance of the object to the origin, with a discount factor. 
The optimal control derived by Bene{\v{s}}, Shepp, and Witsenhausen \cite{BSW1980} is shown to be of a ``bang-bang'' type.  That is, there exists a threshold $c$ such that when the object  is
within $[-c, c]$, it will be idling; and when it is outside $[-c, c]$, the controller  will apply the minimal push needed to bring it back within $[-c, c]$. The controlled dynamics is thus a reflected Brownian motion, with local times 
at $c$ and $-c$ as a result of  the minimal push.
This problem has a number of generalizations; see, for example,  Karatzas~\cite{Karatzas1982}, Karatzas and Shreve~\cite{KS1985}, and Shreve and Soner~\cite{SS1991}.  In particular, Karatzas~\cite{Karatzas1982} derives a similar bang-bang type optimal control  when the $L^2$ distance is relaxed to a class of convex and symmetric functions; see Figure~\ref{figure1}. Due to its simplicity, the fuel follower problem has many applications and has inspired a number of research topics, including reflected stochastic differential equations and semimartingales,  Skorokhod problems, and regularities of fully nonlinear PDEs with gradient constraints.   
See, for instance, Harrison and Williams~\cite{HW1987}, Soner and Shreve~\cite{SS1989}, Varadhan and Williams~\cite{VW1985}, Williams~\cite{WILLIAMS1987}, Dai and Williams~\cite{DW1996}, Kruk~\cite{KRUK2000}, 
Atar and Budhiraja~\cite{AB2006}, Budhiraja and Ross~\cite{BR2006}, Evans~\cite{EVANS1979}, and Hynd~\cite{HYND2010}.

\paragraph{Our work.} In this paper we formulate and analyze an $N$-player stochastic game of the fuel follower problem and  its Mean Field Game (MFG) counterpart.
In the $N$-player game, there are $N$ controllers and $N$ objects with each controller controlling one object. Each controller minimizes her total amount of control and the total distance of her object to {\it the center of the $N$ objects}.  The interaction among the $N$ 
controllers in the  game is to ensure that their own objects closely follow each other's movement.
We derive the Nash Equilibrium (NE) explicitly (Theorem \ref{NE_point_N}). This result is established in two main steps.
  The first step is to derive and analyze a system of Hamilton--Jacobi--Bellman (HJB) equations for the value functions and to establish a verification theorem (Theorem \ref{verification_theorem}) for the game. 
  After finding the solution to the HJB system, the second step is to construct a feedback control via proving the existence of a (unique strong) solution to an associated  Skorokhod problem on an unbounded polyhedron
  with an oblique reflection (Theorem \ref{thm: solution skorokhod}). 
 For the special case of $N=2$, we exploit the symmetric structure to obtain multiple NEs; see Figure \ref{figure_multi_NE}.

We then consider the corresponding MFG with $N \to \infty$, where each controller minimizes her total amount of control and the total distance of her object to {\it the mean position} of all objects. 
Our approach to analyze this MFG is to study directly the two coupled PDEs, the backward parabolic type HJB equation and the forward Kolmogorov equation.
By further exploiting the problem structure, we  derive an NE which  is of a bang-bang type
(Theorem \ref{MFG-SOL}). {The threshold of this bang-bang type NE is state-independent}
 as in the classical fuel follower problem.
We finally discuss the relation between the $N$-player game and the MFG, and show that this NE to the MFG game is an $\epsilon$-NE to the $N$-player game (Theorem \ref{MFG_approximation}).

\paragraph{Our contribution.}

 In general, there are essential technical difficulties in analyzing   $N$-player stochastic games. 
  The underlying
 HJB system is high dimensional,  the existence of its solution is usually hard to analyze, and 
 deriving explicit solutions is even more challenging. Therefore it is in general infeasible to characterize the equilibrium. 
 In the case of the singular control, the HJB equation is even more complex, with additional gradient constraints coming from possible jumps in the control. 
  For MFGs with singular controls, the Hamiltonian for the underlying stochastic control problem diverges and the classical stochastic maximal principle fails.  Moreover, due to the possible non-stationarity of the mean information process, the associated HJB equation is parabolic despite the  infinite-time horizon setting, making it even more difficult to analyze the regularity of the value functions or to derive explicit solutions.


To the best of our knowledge, our work is the first to provide a complete  characterization of the NEs for both the $N$-player stochastic game and the MFG in a singular control setting.    Our explicit solutions are derived for a class of  convex and symmetric functions, without the usual linear-quadratic  structure for MFGs with regular controls in  Bardi~\cite{Bardi2012}, Bardi and Priuli~\cite{BP2014}, Bensoussan, Sung, Yam, and Yung~\cite{BSYY2016}.
 
Moreover, explicit solutions derived in this paper make it possible to directly compare the structural differences between the MFG and the $N$-player game.  It provides useful insights 
 not only for analyzing general $N$-player games but also for proper formulations of MFGs. 
Indeed,  MFGs  may be very different in nature from  $N$-player games:
in the fuel follower problem, the MFG  degenerates to a single-player game in the sense that its NE 
is {threshold-type bang-bang policy where the threshold is state independent} (Proposition \ref{constant_convergence} and Proposition \ref{proposition-convergence}), while the NEs for the $N$-player game are state dependent (Theorem \ref{NE_point_N}). 
The collapse of the MFG to the single player problem (Proposition \ref{constant_convergence}) is a side effect  by the {\it aggregation} in the MFG formulation: players become more {\it anticipative} when they are assumed to be identical. Our analysis also shows that the NE for a
stationary MFG may not be the NE for the corresponding MFG (Remark \ref{remark4}).

{There are also some noteworthy economic insights from our analysis.
For instance,  in the $N$-player game, we show that when the number of players increases, it is more costly for each player to keep track of other players before making  decisions, as  players will intervene more frequently due to the increasing complexity of the game.
Moreover, the bigger the discount factor $\alpha$, the less frequent players will  intervene. (See Remarks \ref{remark1} and \ref{remark2}). 
}

\paragraph{Related work on stochastic games.}

There are a number of papers on non-zero-sum two-player games with singular controls.  By treating one as a controller and the other as a stopper, where the controller minimizes the finite variation process and the stopper decides the optimal time to terminate the game, Karatzas and Li~\cite{KL2011} 
prove the existence of an NE for the game via a BSDE approach.
Hernandez-Hernandez, Simon, and Zervos~\cite{HSZ2015} provide an 
in-depth analysis of the smoothness of the value function and show
that the optimal strategy may not be unique when the controller enjoys a first-move advantage.
Kwon and Zhang~\cite{KZ2015} 
  investigate a game of irreversible investment with singular controls and strategic exit.  
  They characterize  a class of market perfect equilibria and  identify a set of conditions under which the outcome of the game may be unique despite the multiplicity of the equilibria. 
  De Angelis and Ferrari~\cite{DF2016}
  establish the connection between singular controls and optimal stopping times for a non-zero-sum two-player game.  
  Bensoussan and Frehse~\cite{BF2000} consider an $N$-player game with regular controls and obtain the NE via the maximum principle approach. 
    The closest to our problem setting are those of Mannucci~\cite{Mannucci2004} and Hamadene and Mu~\cite{HM2014}. They consider the fuel follower problem in a finite-time horizon with a bounded velocity, and establish the existence of an NE of a two-player game. The former analyzes a strongly coupled parabolic system and the latter uses 
the BSDE technique.  

\paragraph{Related work on MFGs.}

The theory of MFGs has enjoyed tremendous growth since the pioneering works
of Huang,  Malham{\'e}, and Caines~\cite{HMC2006} and Lasry and Lions~\cite{LL2007}. 
The MFG provides a tractable approach to the otherwise challenging $N$-player stochastic games.  
However, except for the general result that the NE of an MFG is an $\epsilon$-Nash equilibrium to the $N$-player game
(see, for instance \cite{HMC2006} and Cardaliaguet,  Delarue, Lasry, and Lions \cite{CDLL2015}  for regular controls and Guo and Joon \cite{GJ2018} for singular controls), 
there are very  limited results on comparing the NE of $N$-player games and MFGs.   The exceptions are 
Carmona, Fouque, and Sun~\cite{CFS2013} for systemic risks, 
Nutz and Zhang \cite{NZ2017} for competition, Lacker and Zariphopoulou~\cite{LZ2017} for portfolio management, and 
\cite{Bardi2012}. 
All these results, however, are with regular controls.  
For MFGs with singular controls,  notions of relaxed stochastic maximal principle
 or relaxed admissible controls have been introduced to establish the existence of optimal controls; see, for instance, Fu and Horst~\cite{FH2017}, Hu, {\O}ksendal, and Sulem~\cite{HOS2014}, and Zhang~\cite{Zhang2012}. \\



\section{$N$-Player Fuel Follower Game}

\subsection{Preliminary: Single Player}\label{fuel_prob}

 
 The classic fuel follower problem is as follows. 
 Consider a probability space 
 $(\Omega, \mathcal{F},$ $\{\mathcal{F}_t\}_{t\geq 0}, \mathbb{P} )$ with a standard Brownian motion $\{B_t\}_{ t \geq 0}$. The position of the object $X_t$ is assumed to be
 \begin{eqnarray}\label{x-dynamic}
X_t = x + B_t +\xi_t^{+} - \xi_t^{-}, \quad X_{0-} = x,
\end{eqnarray}
 where the pair of control $(\xi^+, \xi^-)$ is a non-decreasing, c\`adl\`ag process. 
 The goal  of the controller  is to solve for the value function $v(x)$ of the following optimization problem,
 \begin{eqnarray}\label{fuelproblem}
v(x) = \inf_{(\xi^+,\xi^{-})\in  \mathcal{U}} \mathbb{E} \int_0^{\infty} e^{-\alpha t} \left[  h(X_t)dt + d \check{\xi}_t \right],
\end{eqnarray}
where  the admissible control set $\cal U$ is
\begin{eqnarray*}
 \mathcal{U} &&:= \left\{ (\xi_t^+,\xi_t^-) \given \xi_t^+ \mbox{ and } \xi_t^- \mbox{ are }
\textrm{$\mathcal{F}^{X_{t-}}$-progressively measurable, c\`adl\`ag, non-decreasing}, \right. \\ 
&& \left. \hspace{63pt} \textrm{ with } \mathbb{E} \left[\int_0^{\infty}e^{-\alpha t} d\xi_t^+ \right] <\infty, \mathbb{E} \left[\int_0^{\infty}e^{-\alpha t} d\xi_t^- \right] <\infty,  \textrm{and } \xi_{0-}^+ = \xi_{0-}^-=0   \right\}.
\end{eqnarray*}
Here $\alpha>0$ is a discount factor,  $\{\mathcal{F}^{X_t}\}_{t\ge 0}$ is the natural filtration of $\{X_t\}_{t\ge 0}$,
and $\check{\xi_t} = \xi_t^+ + \xi_t^-$ is the total accumulative amount of controls {up to time t}, called ``fuel usage'', hence the term 
{\it fuel follower problem}. 
In addition, under the assumption 
\begin{itemize}[font=\bfseries]
\item[A1:] The function $h: \mathbb{R} \rightarrow \mathbb{R}$ is assumed to be convex,  symmetric,  twice differentiable, with $h(0) \geq 0$, $h^{\prime \prime}(x)$  decreasing on $\mathbb{R}^{+}$,  and $0 < k< h^{\prime \prime}(x)\leq K$ for some constants $K>k>0$,
\end{itemize}
Problem (\ref{fuelproblem}) is solved (see \cite{BSW1980} and \cite{Karatzas1982}) by  analyzing the associated HJB equation 
\begin{eqnarray}
\label{originalHJB}
\min \left\{ \frac{1}{2} v_{xx}(x)+h(x)-\alpha v(x) , 1-v_x(x),1+v_x(x) \right\} = 0,
\end{eqnarray}
where $v_{x}$ and $v_{xx}$ are the first and second order derivatives of $v$ with respect to $x$, respectively.
The optimal control $\{\xi_t^{*+} ,\xi_t^{*-} \}_{t \ge 0}$ is shown to be of a bang-bang type given by
\begin{eqnarray*}
\xi_t^{*+} &=& \max \left\{0,\max_{0 \leq u \leq t} \left\{-x - {B}_u + \xi_u^{*-}  - c \right\}\right \},\\
\xi_t^{*-} &=& \max \left\{0,\max_{0 \leq u \leq t} \left\{x +{B}_u + \xi_u^{*+}  - c \right\} \right\},
\end{eqnarray*}
where the threshold $c>0$ is the unique positive solution to
\begin{eqnarray}\label{constant_c}
\frac{1}{\sqrt{2 \alpha}} \tanh \left(c \sqrt{2 \alpha} \right) = \frac{p_1^{'}(c)-1}{p_1^{''}(c)},
\end{eqnarray} 
with
  \begin{eqnarray*}
  p_1(x) &=& \mathbb{E} \left[\int_0^{\infty}e^{-\alpha t}h(x+B_t)dt \right] \\
  &=& \frac{1}{\sqrt{2\alpha}} \left( e^{-x \sqrt{2\alpha}} \int_{-\infty}^x h(z)e^{z\sqrt{2\alpha}}dz+e^{x \sqrt{2\alpha}} \int_x^{\infty} h(z)e^{-z\sqrt{2\alpha}}dz \right).
  \end{eqnarray*}
  The corresponding value function $v(x) \in \mathcal{C}^2(\mathbb{R})$ is given by 
\begin{eqnarray}\label{single-solution}
    v(x)=\left\{
                \begin{array}{ll}
                  -\frac{p_1^{\prime \prime}(c) \cosh \left(x\sqrt{2 \alpha}\right)}{2 \alpha \cosh \left(c \sqrt{2 \alpha} \right)} + p_1(x), \qquad \ \  0 \leq & x \leq c,\\
                  v(c)+(x-c), \qquad &x \geq c,\\
                 v(-x), \qquad & x <0.
                \end{array}
              \right.
  \end{eqnarray}
  
In other words, it is optimal for the controller to apply a ``minimal'' push to keep the object within  $[-c,c]$.
Mathematically, the controlled process is a Brownian motion reflected at the boundaries $c$ and $-c$. The minimal
push corresponds to the local time of the Brownian motion  at $c$ and $-c$.   See Figure~\ref{figure1}.

\begin{figure}[H] 
       
     \centering \includegraphics[width=0.4\columnwidth]{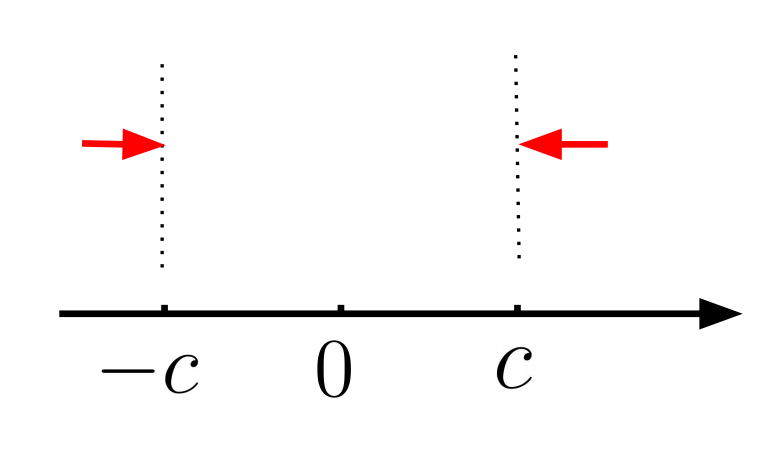}
       
        \caption{
                \label{figure1} 
               Optimal control of the single player problem
        }
\end{figure}


\subsection{$N$-Player  Fuel Follower Game}

Now suppose there are $N$ controllers, with each controller controlling one object.  For simplicity, let us call such a pair of  controller and object  a ``player''.  The game is  for each player to  stay as close as possible to other
players. 

This $N$-player game  can be formulated as follows.  
Let  $\left(X_t^1, \ldots, X_t^N \right)\in \mathbb{R}^N$  be the positions of  players such that for  $i=1, \ldots, N$,
\begin{eqnarray}\label{Ndynamics}
X_t^i =x^i+ B_t^i +  \xi_t^{i,+} -  \xi_t^{i,-}, 
\end{eqnarray}
with $ (X_{0-}^1, \cdots, X_{0-}^N) = (x^1, \cdots,  x^N)=:\pmb{x}$, where $(B_t^1, \ldots, B_t^N)$ is an $N$-dimensional standard Brownian motion on $\mathbb{R}^N$. 
Let  $m_t^{(N)}=\frac{\sum_{i=1}^N X_t^i}{N}$ be the center of these $N$ players at time $t$, with $m_{0-}^{(N)} = \frac{\sum_{i=1}^N x^i}{N}$.  Let $h(X_t^i-m_t^{(N)})$ be  the distance  between player $i$ and the center $ m_t^{(N)}$ at time $t$. The goal of each player $i$ is to minimize, over all admissible controls 
$\left(\xi^1, \ldots, \xi^N \right) \in \mathcal{S}_N$,  the following payoff function 
\[ 
J^i (x^1, \ldots, x^N;\xi^1, \ldots, \xi^N) = \mathbb{E} \int_0^{\infty} e^{-\alpha_i t} \left[  h\left(X^i_t-\rho m_t^{(N)}\right)dt + d \check{\xi}_t^i\right], \label{N-game}  \tag{\textbf{N-player}}
\]
where $\check{\xi}^i ={\xi}^{i,+} +{\xi}^{i,-}  $. 
Here the admissible control set $ \mathcal{S}_N$ is defined as
\begin{eqnarray}\label{S_N}
\begin{aligned}
\mathcal{S}_N :=& \left\{(\xi^1, \ldots, \xi^N) \ \left\vert \   \xi^j = (\xi^{j,+}, \xi^{j,-}) \in \mathcal{U}^j_N, \mathbb{P}\left(d \xi_t^j(\pmb{x}) d\xi_t^i(\pmb{x} )>0 \right)=0, \right.\right.\\
  & \hspace{70pt}  \left.    \mbox{ for any } t>0, \pmb{x} \in \mathbb{R}^N,  i, j \in \{1,\ldots,N\}\mbox{ and } i \neq j  \right\},
\end{aligned}
\end{eqnarray}
with
\begin{eqnarray*}\label{A_N}
\begin{aligned}
 \mathcal{U}^j_N= & \left\{ (\xi_t^{j,+},\xi_t^{j,-})\ \left \vert \  \xi_t^{j,+} \mbox{ and }\xi_t^{j,-} \mbox{ are } \mathcal{F}^{(X^1_{t-},\ldots,X^N_{t-})}\mbox{-progressively measurable, c\`adl\`ag, non-decreasing,} \right.\right.\\
&  \hspace{50pt} \left.{}  \mbox{ with }\mathbb{E} \left[ \int_0^{\infty}e^{-\alpha_j t}d\xi_t^{j,+} \right] <\infty,  \enspace \mathbb{E} \left[ \int_0^{\infty}e^{-\alpha_j t}d\xi_t^{j,-} \right] <\infty, \enspace \xi_{0-}^{j,+}=0, \enspace \xi_{0-}^{j,-}=0 \right\} ,
\end{aligned}
\end{eqnarray*}
where $\alpha_j>0$ is the discount factor for player $j$ and $\{\mathcal{F}^{(X_t^1,\ldots,X_t^N)}\}_{t \geq 0}$ is the natural filtration of $\{(X_t^1,\ldots,X_t^N)\}_{t\geq 0}$.
The condition in Eqn. (\ref{S_N})  
\begin{eqnarray}\label{no_jump_N_player}
 \mathbb{P} \left(d \xi _t^i (\pmb{x})  d \xi _t^ j(\pmb{x})> 0  \right) = 0, \ \   \ \  \mbox{ for any }\pmb{x} \in \mathbb{R}^N, t\ge 0, i \neq j
 \end{eqnarray}
is to  facilitate designing feasible control policies when  controls involve jumps.

{\begin{remark} \label{remark:SPrho}
Mathematically, one may replace the running cost function $h(X_t^i-m_t^N)$ by
$h(X_t^i-\rho m_t^N+\eta)$,  with $\rho\ge0$ indicating the strength of interactions among players as in  \cite{HCM2007} and \cite{HMC2006}. We choose to fix $\rho=1$ and $\eta=0$ for clearer model interpretations for the fuel follower problem. 
Indeed, adding a scaling factor $\rho$ and a constant $\eta$ will not change the derivation of solutions except for minor notational changes. In fact, as will be  shown in Section \ref{section_skorokhod} and Appendix A, the construction of NEs will be simpler when $\rho\ne 1$.

\end{remark}
}

Throughout the paper, unless otherwise specified,  we will for simplicity and without loss of generality  $\alpha_1=\cdots=\alpha_N=\alpha$.   
 (See Section \ref{section:discussion} for further sensitivity analysis with respect to $\alpha$.)


\subsection{Solution to the $N$-Player Game}

There are various criteria to measure the performance of strategies in  stochastic  games. For instance,  Pareto Optimality (PO) and Nash Equilibrium (NE) 
 provide two  distinct views, with NE focusing on stability  and  PO on efficiency. 
An NE framework can be further defined depending on the admissible  strategies, resulting in  open-loop NEs,  closed-loop NEs, and the Markovian NEs. See Carmona~\cite{Carmona2016} for more discussions on these concepts. 

In this paper, we will focus on the Markovian  NE, also known as the closed-loop 
NE with a feedback form, specified  below.

\begin{definition}\label{nash}
A tuple of admissible controls $\pmb{\xi}^* = ({\xi}^{1*},\ldots,{\xi}^{N*}) \in \mathcal{S}_N$ is a Markovian NE of the stochastic game
(\ref{N-game}), if for any $ i=1,\ldots, N$, $\pmb{X}_{0-} = \pmb{x}$, and  any $(\pmb{\xi}^{-i*},{\xi}^i) \in \mathcal{S}_N$,
 the following inequality holds,
\begin{eqnarray*}
J^i \left(\pmb{x}; \pmb{\xi}^{*}\right)  \leq J^i \left(\pmb{x}; (\pmb{\xi}^{-i*}, {\xi}^{i})\right).
\end{eqnarray*} 
Here  strategies ${\xi}^{i*}$ and ${\xi}^{i}$ are deterministic functions of time $t$ and $\pmb{X}_t = (X_t^1,\ldots,X_t^N)$, with the notation  $(\pmb{x}^{-i},y^i):=(x^1,\cdots,x^{i-1},y^i,x^{i+1},\cdots,x^N)$ for any $\pmb{x}\in \mathbb{R}^N$.
 $J^i \left(\pmb{x}; \pmb{\xi}^{*}\right)$ is called the NE value associated with  $\pmb{\xi}^{*}$.

\end{definition}

\subsection{NE Solutions}



The NE solution  will be derived in two steps. 
The first is to derive and analyze the associated HJB system. A verification theorem which provides  sufficient conditions for the NE values will be presented, along with a solution to the HJB system.  The second step
is to construct the  corresponding NEs, by solving an associated Skorokhod problem. 

 \subsection{NE and the HJB System}
First,  

\begin{definition}
[\emph{Action} and \emph{waiting} regions]\label{regions}
Player $i$'s \emph{action region} $ \mathcal{A}_i$ is defined as
$$ \mathcal{A}_i:= \left\{\left.\pmb{x} \in \mathbb{R}^N \right\vert d \xi^i(\pmb{x}) \neq 0 \right\},$$ 
and her \emph{waiting region} is $\mathcal{W}_i  = \mathbb{R}^N \setminus  \mathcal{A}_i$.  Denote $\mathcal{A}^{-i} = \cup_{j \neq i}  \mathcal{A}_j $ 
and $\mathcal{W}_{-i} = \cap_{j \neq i}  \mathcal{W}_j$. 
\end{definition}

%


Next, 
a simple heuristic conditional argument via the Dynamic Programming Principle  leads to the following  HJB system.

{Given $\mathcal{A}_i \cap  \mathcal{A}_j = \emptyset$, for any $ i \neq j$,}

\begin{eqnarray*}\label{HJB-N}
    \mbox{(HJB-N)} \left\{
                \begin{array}{ll}
             &   \min \left\{-\alpha w^i + h\left(\frac{N-1}{N} \left(x^i- \frac{\sum_{j \neq i} x^j}{N-1} \right) \right) + \frac{1}{2} \left(\sum_{j=1}^N w^i_{x^j x^j} \right),   1-w^i_{x^i},1+w^i_{x^i} \right\} = 0,      \\
             &     \hspace{216pt}    \mbox{     for any } \pmb{x} \in \mathcal{W}_{-i} \label{fv-N-player-hjb-1} , \\
           
 &{w^i_{x^j} =0,  \hspace{173pt} \mbox{ for any } \pmb{x} \in  \mathcal{A}_{j}, \mbox{ for any } j \neq i } .

                \end{array}
              \right.
  \end{eqnarray*}

\vskip 2mm
The derivation of  {(HJB-N)} can be  illustrated with the case of $N=2$.
In this case,  if $ (x^1,x^2) \in \mathcal{A}_2$,  $\Delta \xi^{2*} \neq 0$. By the definition of NE, player one 
is  not expected to suffer a loss as otherwise she will have  incentives to take actions. {Therefore, 
$w^1(x^1,x^2) = w^1(x^1, x^{2}+\Delta \xi^{2*,+}-\Delta \xi^{2*,-})$, letting $\Delta \xi^{2*,\pm}\rightarrow 0$, we have $w_{x^2}^1=0$ in $\mathcal{A}_2$.}
 If $ (x^1,x^2) \in \mathcal{W}_2$,  $\Delta \xi^{2*} = 0$, then the control problem for player one becomes a classical single player control problem. Therefore, $w^1(x^1,x^2)$ satisfies
$$\min \left\{-\alpha w^1 +h\left(\frac{x^1-x^2}{2} \right) + \frac{1}{2}\left(w^1_{x^1 x^1} + w^1_{x^2 x^2}\right),1-w^1_{x^1},1+w^1_{x^1}\right\} = 0 \mbox{ in } \mathcal{W}_2.$$ 
Here $-\alpha w^1 +h\left(\frac{x^1-x^2}{2} \right) + \frac{1}{2}\left(w^1_{x^1 x^1} + w^1_{x^2 x^2}\right)=0 $ corresponds to  $\Delta \xi^{1*}=0$, $1-w^1_{x^1}=0$ corresponds to  $\Delta \xi^{1*,+}> 0$, and  $1+w^1_{x^1}=0$  corresponds to  $\Delta \xi^{1*,-}> 0$. Finally, 
$\mathcal{A}_1 \cap \mathcal{A}_2 =\emptyset$ ensures Eqn. (\ref{no_jump_N_player}).

 \vskip 1.5mm

Based on the above HJB system, the following sufficient conditions for an NE can be established.
\begin{theorem}[Verification theorem] \label{verification_theorem}
 For any $i=1, \ldots, N$,  suppose  ${\xi}^{i*} \in  \mathcal{U}^i_N$  and the corresponding $w^i(.) = J^i(.;\pmb{\xi}^*)$
 satisfies the following
 \begin{enumerate}[label=(\roman*)]
 \item \label{item1} $\pmb{\xi}^* := ({\xi}^{1*},\ldots,{\xi}^{N*}) \in \mathcal{S}_N$,
 \item \label{item2} 
 \begin{eqnarray} \label{hjb-n}
 \min  \left\{-\alpha w^i +h \left(\frac{N-1}{N} \left(x^i- \frac{\sum_{j \neq i} x^j}{N-1} \right) \right) + \frac{1}{2} \sum_{j=1}^N w^i_{x^j x^j},  1-w^i_{x^i},1+w^i_{x^i} \right\} = 0,  
 \end{eqnarray}
 for any $\pmb{x} \in \overline{\mathcal{W}_{-i}}$, and
 $$w^i_{x^j}(\pmb{x})= 0,$$   for any  $\pmb{x} \in  \mathcal{A}_{j}$. 
  \item \label{item3} (Transversality Condition.) \  \  
$ 
\limsup_{T \rightarrow \infty} \mathbb{E}[e^{-\alpha T} w^i(\pmb{X}_T)] = 0,$
 \item \label{item4} $w^i (\pmb{x}) \in  \mathcal{C}^2(\overline{\mathcal{W}_{-i}})$,
  \item  \label{item7_add} $w_{x^j}^i (\pmb{x})$ is bounded in $\overline{\mathcal{W}_{-i}}$, for any $j=1,2,\cdots,N$,
 \item \label{item5} there exists a convex function $u^i(\pmb{x})\in  \mathcal{C}^2(\mathbb{R}^N)$ such that $u^i(\pmb{x}) = w^i(\pmb{x})$  on $\overline{\mathcal{W}_{-i}}$, 
 \item \label{item6}   $\mbox{ for any } \xi^i \in  \mathcal{U}_{N}$ such that $(\pmb{\xi}^{-i*},\xi^{i}) \in \mathcal{S}_N$, the controlled dynamic $(\pmb{X}_t^{-i*},X_t^i)$ is in $ {\mathcal{W}_{-i}}$ $\mathbb{P}$-a.s. at any time $t$.
 \end{enumerate} 
Then $\pmb{\xi}^*$ is an NE with value $w^i$. 
\end{theorem}


\begin{proof}


Given any $\xi^i \in  \mathcal{U}^i_N$ such that $(\pmb{\xi}^{-i*},\xi^{i}) \in \mathcal{S}_N$, fixing the control $(\xi_t^{i,+} , \xi_t^{i,-})$ such that 
\begin{eqnarray*}
X_t^i &=& x^i+ B_t^i + \xi_t^{i,+}- \xi_t^{i,-}, \\
X_t^{j*} &=& x^j +B_t^j + \xi_t^{j*,+}- \xi_t^{j*,-}, \ \    j \neq i.
\end{eqnarray*}
Applying the It\^{o}-Tanaka-Meyers formula (Theorem 14.3.2 in \cite{CE2015}) to  $e^{-\alpha t} u^i(\pmb{X}_t^{-i*},X_t^{i})$
yields
\begin{eqnarray*}
&&\mathbb{E}\left[e^{-\alpha T} u^i(\pmb{X}_T^{-i*},X_T^{i}) \right] - u^i(x^1,x^2,\ldots,x^N) \\
=&& \mathbb{E} \left[ \int_0^T e^{-\alpha t} \left( \frac{1}{2}\sum_{j=1}^N u^i_{x^j x^j}(\pmb{X}_t^{-i*},X_t^{i}) -\alpha u^i(\pmb{X}_t^{-i*},X_t^{i}) \right) dt \right]\\ 
&+& \mathbb{E} \left[ \int_{[0,T)} e^{-\alpha t} \left( (u^i_{x^i}(\pmb{X}_t^{-i*},X_t^{i}) d \xi_t^{i,+} - u^i_{x^i}(\pmb{X}_t^{-i*},X_t^{i}) d \xi_t^{i,-}) \right) \right]\\
&+& \mathbb{E} \left[ \sum_{0\leq t <T } e^{-\alpha t}  \left( \Delta u^i \left(\pmb{X}_t^{-i*},X_t^{i}\right)-\nabla u ^i \left(\pmb{X}_t^{-i*},X_t^{i}\right) \cdot \Delta (\pmb{X}^{-i*}_t,X_t^{i} )\right) \right]\\
&+&\mathbb{E} \int_0^T e^{{-\alpha}t} \left( \sum_{j=1}^N u^i_{x^j} (\pmb{X}_t^{-i*},X_t^{i}) d B_t^j \right).
\end{eqnarray*}
Note that \ref{item6} implies that with control $(\pmb{\xi}^{-i*},\xi^{i})\in \mathcal{S}_{N}$,   $(\pmb{X}_t^{-i*},X_t^{i}) \in {\mathcal{W}_{-i}}$,  $\mathbb{P}$-a.s..  By conditions \ref{item7_add} and  \ref{item5}, $u^i_{x^j}$ is bounded on $ \overline{\mathcal{W}_{-i}}$ for any $1\le j\le N$, therefore
$\int_0^T e^{{-\alpha}t} \left( \sum_{j=1}^N u^i_{x^j} (\pmb{X}_t^{-i*},X_t^{i}) d B_t^j \right)$ is square integrable, hence
a uniformly integrable martingale. Now conditions \ref{item2},  \ref{item4}, \ref{item7_add}, and  \ref{item5} suggest
\begin{eqnarray*}\label{verify_last}
e^{-\alpha T} \mathbb{E}[w^i(\pmb{X}_T^{-i*},X_T^{i})] + \mathbb{E}\int_{0}^T e^{-\alpha t} \left[ h\left( \frac{N-1}{N}\left(X_t^i - \frac{\sum_{j \neq i}X_t^{j*}}{N-1} \right)\right) dt + d\check{\xi}_t^{i}\right]  \geq  w^i(x^1,\ldots,x^N).
\end{eqnarray*}
Taking $T \rightarrow \infty$,  the transversality condition \ref{item3} implies 
\begin{eqnarray}\label{lower-bond-1}
 w^i \left(x^1,\ldots,x^N \right) \leq J^i \left(x^1,\ldots,x^N;\pmb{\xi}_t^{-i*},\xi_t^{i} \right),
 \end{eqnarray}
 for any ${\xi}^i$ such that $\left(\pmb{\xi}_t^{-i*},\xi_t^{i} \right) \in \mathcal{S}_N$.
\end{proof}

The next step is to solve the HJB system, with a focus on a threshold-type solution.
That is,  there exists a constant $c_N>0$ (to be determined) such that the action region $\mathcal{A}_i$ and the waiting $\mathcal{W}_i$ of player $i$ can be  decomposed into
\begin{eqnarray}\label{two_regions}
\begin{aligned}
 \mathcal{A}_i = \left\{E_i^{-} \cup E_i^{+} \right\}\cap Q_i, \quad \mathcal{W}_i = \mathbb{R}^N /  \mathcal{A}_i ,
\end{aligned}
\end{eqnarray}
where
\begin{eqnarray}\label{3-regions}
\begin{aligned}
E_i^{-} &= \left\{ \left(x^1,\cdots,x^N \right) \in \mathbb{R}^N \middle| \enspace x^i-\frac{\sum_{j \neq i}x^j}{N-1} \leq -c_N \right\},\\
E_i^{+} &= \left\{ \left(x^1,\cdots,x^N \right) \in \mathbb{R}^N \middle| \enspace x^i-\frac{\sum_{j \neq i}x^j}{N-1} \geq c_N \right\},
\end{aligned}
\end{eqnarray}
with the partition
\begin{eqnarray*}\label{region-i-charge}
Q_i &=& \left\{ \pmb{x} \in \mathbb{R}^N  \middle| \quad  \left|x^i - \frac{\sum_{j \neq i} x^j}{N-1}\right|  \geq  \left|x^k- \frac{\sum_{j \neq k} x^j}{N-1} \right|, \mbox{ for any } k < i ; \right.\\ 
&&\hspace{56pt}\left.{}  \left|x^i - \frac{\sum_{j \neq i} x^j}{N-1}\right|  >  \left|x^k- \frac{\sum_{j \neq k} x^j}{N-1} \right| , \mbox{ for any } k > i  \right\}.
\end{eqnarray*}
Note the modification of the action region $\mathcal{A}_i$  by $Q_i$ is to avoid simultaneous jumps by multiple players. 
By definition of $Q_i$, in the event of multiple players in the ``action region'', the player who is the farthest away from the center intervenes first; in the event that multiple players have the same largest distance to the center, the player with the  biggest  index intervenes. 

 
Now it is easy to check that 
\begin{itemize}
\item $\cup_{i=1}^N Q_i = \mathbb{R}^N$, $Q_i$ is a convex cone for any $i=1,\ldots,N$,
\item $\mathcal{W}_i \neq \emptyset$, for any $i =1,\ldots,N$,
\item $ \mathcal{A}_i \cap  \mathcal{A}_j = 0$, for  all $i \neq j$.
\end{itemize}

%
%

Now, a candidate function $w^i(\pmb{x}) \in  \mathcal{C}^2(\overline{\mathcal{W}_{-i}})$ should satisfy the following three properties:
First, $w^i(\pmb{x})$ is symmetric on $x^i = \frac{\sum_{j \neq i} x^j}{N-1}$ such that 
\begin{eqnarray}\label{condition-sym}
w^i_{x^i} \left( \pmb{x}^{-i},\frac{\sum_{j \neq i} x^j}{N-1} \right) =0.
\end{eqnarray}
Second, if $0 \leq x^i-\frac{\sum_{j \neq i} x^j}{N-1}<c_N$, then $w^i(\pmb{x})$ solves
\begin{eqnarray}\label{condition-donothing}
\alpha w^i (\pmb{x})  = h \left( \frac{N-1}{N} \left( x^i- \frac{\sum_{j \neq i} x^j}{N-1} \right) \right)+ \frac{1}{2} \sum_{j=1}^N w^i_{x^j,x^j}(\pmb{x}).
\end{eqnarray}
Third, if  $x^i - \frac{\sum_{j \neq i} x^j}{N-1} \geq   c_N$, then player $i$ jumps by a distance of $x^i - \frac{\sum_{j \neq i} x^j}{N-1} -   c_N$. Combined, 
\begin{eqnarray}\label{condition-push}
w^i(\pmb{x}) =x^i - \frac{\sum_{j \neq i} x^j}{N-1} -   c_N + w^i  \left( \pmb{x}^{-i}, \frac{\sum_{j \neq i} x^j}{N-1} +  c_N \right).  
\end{eqnarray}
The general solution satisfying both (\ref{condition-donothing}) and (\ref{condition-sym}) is given by
\begin{eqnarray*}
w^i(\pmb{x}) = B \cdot \cosh \left( \sqrt{\frac{2(N-1) \alpha}{N}} \left(x^i- \frac{\sum_{j \neq i} x^j}{N-1} \right) \right) + p_N \left(x^i-\frac{\sum_{j \neq i} x^j}{N-1} \right),
\end{eqnarray*}
with 
\begin{eqnarray} \label{p_N}
p_N(x) = \mathbb{E} \left[ \int_0^{\infty} e^{-\alpha t} h\left( \frac{N-1}{N} \left(x +\sqrt{\frac{N}{N-1}} B_t \right) \right)dt \right].
\end{eqnarray}
Here
$p_N(x)$  is a particular solution to (\ref{condition-donothing}) and derived from  the cost of ``doing nothing'', and 
 $B$ is constant yet to  be  determined.

 Now matching the values of $w_{x^i}(\pmb{x})$ and $w_{x^i,x^i}(\pmb{x})$ along $x^i = \frac{\sum_{j \neq i} x^j}{N-1} +c_N$ determines $c_N$ and $B$:  $c_N$ is the unique positive solution to
\begin{eqnarray}\label{c2}
\frac{1}{\sqrt{\frac{2(N-1)\alpha}{N}}} \tanh \left(c \sqrt{\frac{2(N-1)\alpha}{N}} \right) = \frac{p'_N(c)-1}{p''_N(c)},
\end{eqnarray}
and \begin{eqnarray*}
B =   -\frac{p_N^{\prime \prime}(c_N)}{\frac{2(N-1) \alpha}{N} \cosh \left( c_N\sqrt{\frac{2(N-1)\alpha}{N}} \right)} .
\end{eqnarray*}
Finally, define

 \begin{eqnarray*}\label{u-i}
    u^{i}(x^1,\ldots,x^N) = \left\{
                \begin{array}{ll}   
          u^{i} \left(x^1,\ldots,\frac{\sum_{j \neq i}x^j}{N-1}-c_N,\ldots,x^N \right) - c_N-x^i+\frac{\sum_{j \neq i}x^j}{N-1}, \hspace{20pt}   \pmb{x} \in E_i^{-}  ,\\
                -\frac{p_N^{\prime \prime} \left(c_N\right) \cosh \left(\sqrt{\frac{2(N-1) \alpha}{N}} \left(x^i-\frac{\sum_{j \neq i}x^j}{N-1}\right)\right)}{\frac{2(N-1) \alpha}{N} \cosh \left(c_N\sqrt{\frac{2(N-1)\alpha}{N}} \right)} + p_N \left(x^i-\frac{\sum_{j \neq i}x^j}{N-1} \right), \\ 
                  \hspace{240pt}  \pmb{x} \in  \{E_i^{+} \cup E_i^{-}\} ^c,\\
         x^i -\frac{\sum_{j \neq i}x^j}{N-1}- c_N  +u^i \left(x^1,\ldots,\frac{\sum_{j \neq i}x^j}{N-1}+c_N,\ldots,x^N \right) ,
     \hspace{20pt}      \pmb{x} \in E_i^{+}.
                \end{array}
              \right.
 \end{eqnarray*}
Then it is easy to check that $u^i \in \mathcal{C}^2(\mathbb{R}^N)$ and the candidate solution $w^i$  satisfies (HJB-N) and  Theorem \ref{verification_theorem}.

\subsection{NE and the Skorokhod Problem (SP)} \label{section_skorokhod}
Given the NE solution to the $N$-player game,  the corresponding NE 
can be constructed by finding a solution to an associated SP on an unbounded polyhedron and with a constant oblique reflection  on each face.
First,  define $\mathcal{CW}$  the common waiting regions of all players as
 \begin{eqnarray}\label{nonaction_region}
\mathcal{CW} &:=& \left\{ \pmb{x} \in \mathbb{R}^N \enspace \middle|  \enspace \left|x^i-\frac{\sum_{j \neq i}x^j}{N-1} \right| < c_N, \mbox{ for any } i=1,\ldots,N \right\} \\
&=&  \left\{\pmb{x} \in \mathbb{R}^N \enspace \middle| \enspace \pmb{n}_j \cdot \pmb{x} >-c_N \sqrt{\frac{N}{N-1}}, \mbox{ for } j=1,\ldots, 2N \right\} \nonumber  \\
&=& \cap_{i=1}^N ({E}_i^-\cup {E}_i^+)^c, \nonumber
\end{eqnarray}
with the normal direction of each face  given by
\begin{eqnarray}
\begin{aligned}
\pmb{n}_i &= \frac{\sqrt{N-1}}{\sqrt{N}} \left( -\frac{1}{N-1}, \cdots, -\frac{1}{N-1}, 1,-\frac{1}{N-1}, \cdots, -\frac{1}{N-1} \right), \\
\pmb{n}_{i+N} &=-\pmb{n}_i.
 \end{aligned}
\end{eqnarray}
where 1 is in the $i^{th}$ position of $\sqrt{\frac{N}{N-1}}\pmb{n}_i$.
Note that  $\mathcal{CW} $ is an unbounded polyhedron with all of its $2N$ boundaries parallel to the direction $(1,1,\cdots,1).$ 

For $j=1, \cdots, 2N$, define the $2N$ faces of ${\mathcal{CW}}$ 
\begin{eqnarray} \label{face}
F_j = \{\pmb{x} \in  \partial \mathcal{CW} \ \vert \ \pmb{n}_j\cdot \pmb{x} =-c_N\},
\end{eqnarray}
and 
\begin{equation}\label{di}
\pmb{d}_i =(0, \cdots, 1, \cdots, 0),  \ \  \pmb{d}_{i+N} = -\pmb{d}_i, \quad i =1,\ldots, N,
\end{equation}
such that 
$\pmb{d}_j \cdot \pmb{n}_j =\frac{\sqrt{N-1}}{\sqrt{N}},$
where $1$ is in the $i^{th}$ position of $\pmb{d}_i$.

Now, the  NE of (\ref{N-game}) can be fully characterized by the solution to the SP with the data
 $(\pmb{x},  \mathcal{CW},
(\pmb{d}_1, \cdots, \pmb{d}_{2N}), \{\pmb{B}_t\}_{t\ge 0})$. (See Appendix A for more background materials.)

\begin{theorem}\label{thm: solution skorokhod}
There exits a unique strong solution to  SP with the data  $(\pmb{x},  \mathcal{CW},
(\pmb{d}_1, \cdots, \pmb{d}_{2N}), \{\pmb{B}_t\}_{t\ge 0})$ defined in (\ref{nonaction_region}) and
(\ref{di}). More precisely, the reflected process $\pmb{X}^*_t$ with $\pmb{X}^*_0=\pmb{x} \in \overline{\mathcal{CW}}$ is defined as
\begin{eqnarray*}
X_t^{i*} = x^i+B_t^i+\int_0^t 1_{\{\pmb{X}_s^*\in F_i\}} d \eta^i(s) -\int_0^t 1_{\{\pmb{X}_s^*\in F_{i+N}\}} d \eta^i(s), \quad i=1,2,\cdots,N,
\end{eqnarray*}
where $\eta^j(t)$ is a non-decreasing process with $\eta^j(0)=0$. Moreover, if $\pmb{x} \notin F_k \cap F_j$ for any $k \neq j, k, j=1, 2, \cdots, 2N,$
\begin{eqnarray}\label{no-edgy}
\mathbb{P}(\pmb{X}^*_t \notin F_k \cap F_j \mbox{ for any } k \neq j, t \geq 0)=1.
\end{eqnarray}
\end{theorem}

The idea to prove Theorem \ref{thm: solution skorokhod} is to show first 
 the existence of a weak solution to the SP and next the uniqueness of the strong solution 
 to  the SP.
Then  according to Corollary 3.23 in Karatzas and Shreve~\cite{KS2012} and Proposition
1 in Engelbert \cite{Engelbert1991}, there exists a unique strong solution to the SP.
The  existence of a weak solution to the SP is straightforward, following  ~\cite{DW1996}.
 The uniqueness of a strong solution is established by
 extending the result of  Dupuis and Ishii~\cite{DI1993} on a bounded polyhedron to  an unbounded one,  via the localization technique.  
 Moreover, the reflection vectors $(\pmb{d}_1,\cdots,\pmb{d}_{2N})$ satisfy the \textit{skew symmetry} condition for the polyhedron $\mathcal{CW}$ according to \cite{WILLIAMS1987}, hence  an additional localization argument shows that  (\ref{no-edgy}) holds.
 The detailed proof is  provided in Appendix A.

\subsection{Extended Mapping to $\mathbb{R}^N \setminus \overline{\mathcal{CW}}$}

Up to now the NE is derived  when   $\pmb{x} \in \overline{\mathcal{CW}}$. 
 When $\pmb{x} \in \mathbb{R}^N \setminus \overline{\mathcal{CW}}$,  the NE would be  to jump sequentially to some point $\hat{\pmb{x}} \in \partial \overline{\mathcal{CW}}$, and afterwards continues according to the SP with data  $(\hat{\pmb{x}},  \mathcal{CW},
(\pmb{d}_1, \cdots, \pmb{d}_{2N}), \{\pmb{B}_t\}_{t\ge 0})$ where $\hat{\pmb{x}} \in \overline{\mathcal{CW}}$. 
 
 Algorithm $\bf 1$ describes how players sequentially jump to $\overline{\mathcal{CW}}$.
 In order to show that this algorithm is well defined, one needs to make sure that such jumps stop  in finite steps or  converge to a limit point on $\hat{\pmb{x}} \in \partial \overline{\mathcal{CW}}$, and that the total distance of such sequential jumps is bounded. The detailed argument is given  in Appendix B, with the illustration of Figure \ref{figure_jumps}.

 \begin{algorithm}[H]
\caption{Policy: Sequential jumps when $\pmb{x} \notin \overline{\mathcal{CW}}$. }\label{initial_jump}
\begin{algorithmic}[1]
\Procedure{Sequential}{$\pmb{x}$}
\State Define mapping,
\begin{eqnarray}\label{mapping_alg_1}
\begin{aligned}
i &= \pi(\pmb{y}) \quad \mbox{when} \quad \pmb{y} \in  \mathcal{A}_i,\\
\emptyset &= \pi(\pmb{y})  \quad \mbox{when} \quad  \pmb{y} \in \overline{\mathcal{CW}}.
\end{aligned}
\end{eqnarray}
\State $\hat{\pmb{x}} \gets \pmb{x}$, $k \gets 0$
\While{$\pi(\hat{\pmb{x}})\not= \emptyset$}
\State  $\lambda^* \gets \arg\min \left\{ \lambda >0 \ \left\vert \ \hat{\pmb{x}}+ \lambda \pmb{e}_{\pi(\hat{\pmb{x}})} \in  \partial E^{-}_{\pi(\hat{\pmb{x}})}  \mbox{ or }  \hat{\pmb{x}} - \lambda \pmb{e}_{\pi(\hat{\pmb{x}})} \in  \partial E^{+}_{\pi(\hat{\pmb{x}})}  \right.\right\}$ \Comment{ $e_j$ is a unit vector in $\mathbb{R}^N$ with $j$th component to be 1}
\If{$\hat{\pmb{x}}+ {\lambda^*} \pmb{e}_{\pi(\hat{\pmb{x}})} \in  \partial E^{-}_{\pi(\hat{\pmb{x}})}$}
\State $\pmb{\nu}_0 \gets \pmb{e}_{\pi(\hat{\pmb{x}})}$ 
\Else
\State $\pmb{\nu}_0 \gets - \pmb{e}_{\pi(\hat{\pmb{x}})}$ 
\EndIf
\State $\hat{\pmb{x}} \gets \hat{\pmb{x}} + \lambda^* \pmb{\nu_0}$ \Comment{Control of player $\pi(\hat{\pmb{x}})$}
\State $ \pmb{x}_{k}\gets \hat{\pmb{x}}$
\State $k \gets k+1$
\EndWhile\label{euclidendwhile}
\State \textbf{return} $\hat{\pmb{x}}, \{\pmb{x} _{k}\}$\Comment{$\hat{\pmb{x}} \in \partial {\mathcal{CW}}$}
\EndProcedure
\end{algorithmic}
\end{algorithm}

Note  that this algorithm gives an $\epsilon$-NE in finite steps. In the case that the starting point is in the intersection of faces, a small perturbation in the algorithm and in the NE value will recover the case of $\pmb{x}\in \mathcal{CW}$.
In summary,

\begin{theorem}[NE for the $N$-player game] \label{NE_point_N} 
Under Assumption {\bf A1},  a Markovian NE for  game (\ref{N-game})  is given by 
\begin{eqnarray}\label{opt-control-n-player}
\begin{aligned}
\xi_t^{i*,+} &= \Delta_0^{i*,+} + \int_0^t \textbf{1}_{ \left\{\pmb{X}_s^* \in F_i \right\} }d \eta^i(s) ,\\
\xi_t^{i*,-} &= \Delta_0^{i*,-} + \int_0^t \textbf{1}_{ \left\{\pmb{X}_s^* \in F_{i+N} \right\}} d \eta^{i+N}(s),
\end{aligned}
\end{eqnarray}
where $\mathcal{CW}$ is given in (\ref{nonaction_region}), $\pmb{X}_t^*$ is the controlled dynamic with $\pmb{X}_0^* = \hat{\pmb{x}} = \pmb{x} + \pmb{\Delta_0}^{*,+} -  \pmb{\Delta_0}^{*,-} \in \overline{\mathcal{CW}},$
with $ \eta^j(t) = \int_0^t \textbf{1}_{ \{\pmb{X}_s^* \in F_j \} } d \eta^j(s) $ and $\eta^j(0)=0$ $(j=1,2,\cdots,2N)$, the  jumps at time $0$ are
\begin{eqnarray}\label{jump_pm}
\begin{aligned}
\Delta_0^{i*,+}  &=& \sum_{k}  \textbf{1}_{ \left\{ \pmb{x}_{k} \in  \mathcal{A}_i \right\}} \left(x_{k+1}^i - x_{k}^i \right)_{+},\\
\Delta_0^{i*,-} &=& \sum_{k} \textbf{1}_{\left\{ \pmb{x}_{k} \in  \mathcal{A}_i \right\}} \left(x^i_{k}-x^i_{k+1} \right)_{+},
\end{aligned}
\end{eqnarray}
with $\{\pmb{x}_{k}\}$ the sequence of jumps prescribed by Algorithm {\bf 1}.

The corresponding NE value   $v^i(x^1,\ldots,x^N) := J^i (x^1,\ldots,x^N;\pmb{\xi}^{*})$
is given by
\begin{eqnarray}\label{n-value-i}
    v^{i}(x^1,\ldots,x^N) =
     \left\{
                \begin{array}{ll}
                 v^{i} \left(x^1,\ldots,x^{j-1},\frac{\sum_{k \neq j}x^k}{N-1}-c_N,x^{j+1},\ldots,x^N \right) ,  \\
        \hspace{160pt}   \pmb{x} \in E_j^-\cap  \mathcal{A}_{j} , \mbox{ for any } j \neq i,\\
      v^{i} \left(x^1,\ldots,\frac{\sum_{j \neq i}x^j}{N-1}-c_N,\ldots,x^N \right) - c_N-x^i+\frac{\sum_{j \neq i}x^j}{N-1},  \\
            \hspace{230pt}   \pmb{x} \in E_i^{-} \cap  \overline{\mathcal{W}}_{-i} ,\\
                -\frac{p_N^{\prime \prime}(c_N) \cosh \left(\sqrt{\frac{2(N-1) \alpha}{N}}\left(x^i-\frac{\sum_{j \neq i}x^j}{N-1}\right)\right)}{\frac{2(N-1) \alpha}{N}  \cosh \left(c_N\sqrt{\frac{2(N-1)\alpha}{N}}\right)} + p_N \left(x^i-\frac{\sum_{j \neq i}x^j}{N-1}\right),   \\
    \hspace{190pt}                 \pmb{x}   \in  (E_i^{-}\cup E_i^{+} )^c\cap  \overline{\mathcal{W}}_{-i},\\
    x^i -\frac{\sum_{j \neq i}x^j}{N-1}- c_N  +v^i \left(x^1,\ldots,\frac{\sum_{j \neq i}x^j}{N-1}+c_N,\ldots,x^N \right) , \\
      \hspace{230pt}    \pmb{x} \in E_i^{+} \cap \overline{\mathcal{W}}_{-i} ,\\
           v^{i} \left(x^1,\ldots,x^{j-1},\frac{\sum_{k \neq j}x^k}{N-1}+c_N,x^{j+1},\ldots,x^N \right) ,  \\
     \hspace{160pt}          \pmb{x} \in E_j^+\cap  \mathcal{A}_{j} , \mbox{ for any } j \neq i.      
                \end{array}
              \right.
 \end{eqnarray}
Here $E_i^{+}$, $E_i^{-}$  are given in (\ref{3-regions}), and $ \mathcal{A}_i$ and $\mathcal{W}_i$  defined in (\ref{two_regions}).


\end{theorem}

Figure (\ref{fig:region_3d}) shows the region partition when $N=3$.
$\mathcal{CW}$, the unbounded polytope, is surrounded by the action regions $\mathcal{A}_i$, $i=1, 2, 3$.
Figure (\ref{fig:region_screenshot}) shows the action region $\mathcal{A}_1$ of player one and the common waiting region $\mathcal{CW}$ of all players. 
\begin{figure}[H]
\centering
\begin{subfigure}{.33\textwidth}
  \centering
  \includegraphics[width=1.0\linewidth]{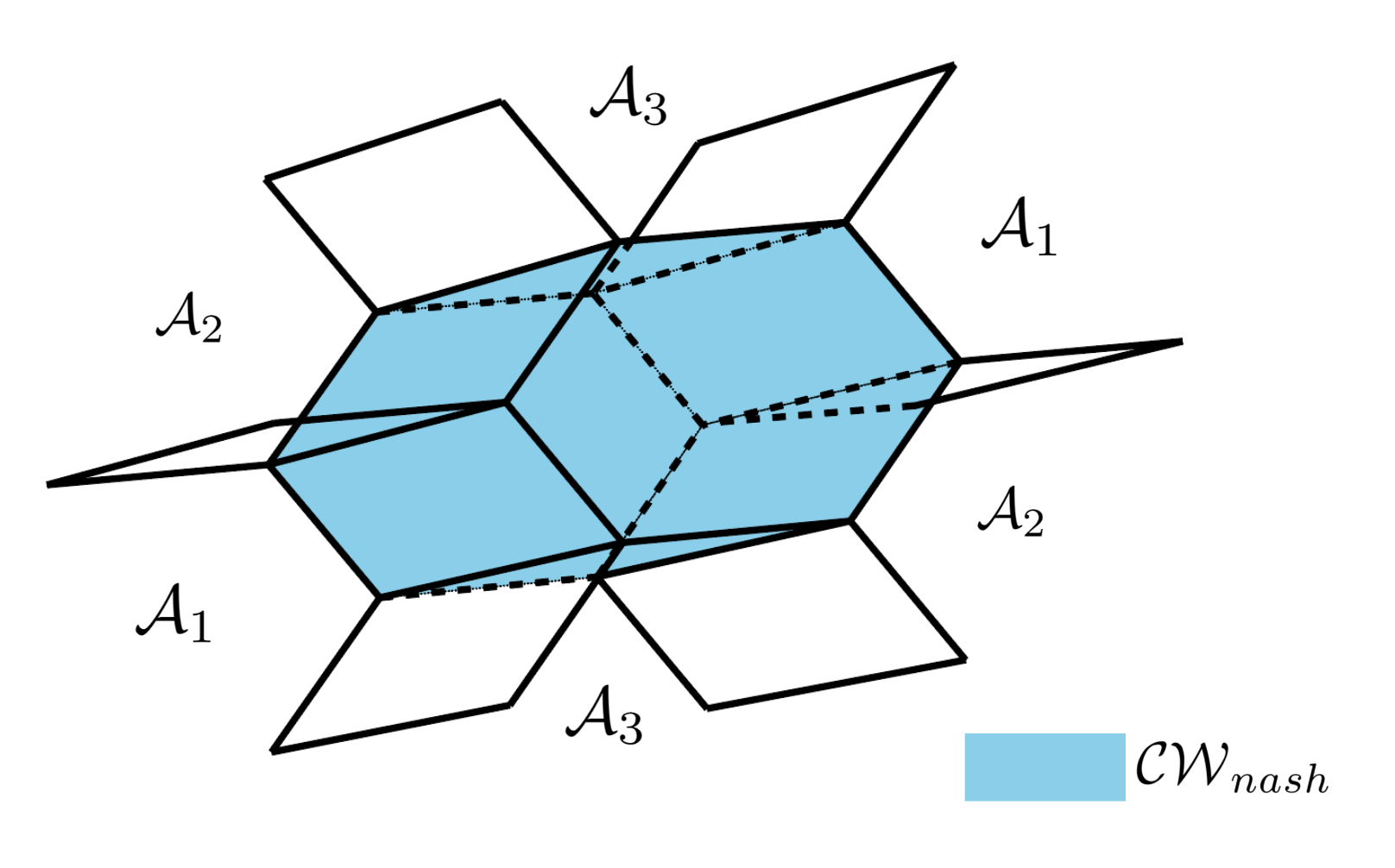}
  \caption{$\mathcal{CW}$ and $\mathcal{A}_i$ ($i=1,2,3$)}
  \label{fig:region_3d}
\end{subfigure}
\begin{subfigure}{.33\textwidth}
  \centering
  \includegraphics[width=1.0\linewidth]{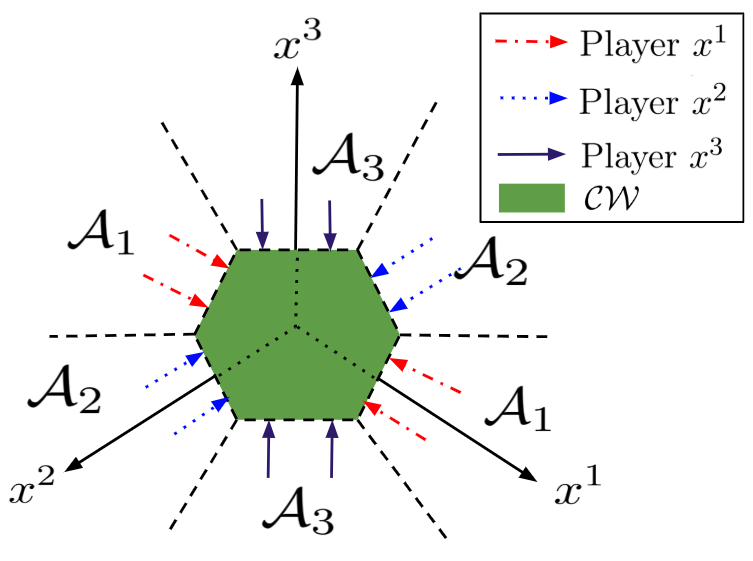}
  \caption{$\mathcal{A}_1$ and $\mathcal{W}_{1}$,\\a bird's-eye view from $(1,1,1)$\\to $(0,0,0)$}
  \label{fig:region_screenshot}
\end{subfigure}
\caption{Region partition when $N=3$}
\label{figure3}
\end{figure}



\section{MFG for the Fuel Follower Problem}

\label{section: MFG}

Take  $N$  identical, rational, and interchangeable players, whose initial positions are random in $\mathbb{R}^N$. 
Let $N \rightarrow \infty$,  the MFG for the fuel follower problem is to find {a closed-loop control in feedback form of}
\begin{eqnarray}\label{game}
\begin{aligned}
v(x) =& \inf_{(\xi^{+},\xi^{-}) \in  \mathcal{U}_{\infty}} J_{(\infty)} (x;\xi_t^{+},\xi_t^{-}) \\
=& \inf_{(\xi^{+},\xi^{-}) \in  \mathcal{U}_{\infty}} \mathbb{E} \int_0^{\infty} e^{-\alpha t} \left[ h(X_t -m_t)dt + d \check{\xi_t} \vert X_{0-}=x \right],\\
 \mbox{such that} & \qquad d X_t = d B_t + d \xi_t^ {+} - d \xi_t^{-}, \\
 &\hspace{20pt} \ \ X_{0-}\sim \mu_{0-},\,\, m_{0-} =  \int x \mu_{0-}(dx),
\end{aligned}
\end{eqnarray}
{where $\mu_t = \lim_{N \rightarrow \infty} \frac{ \sum_{i=1}^N \textbf{1}_{\{X^i_t \}}}{N}$ is the distribution of $X_t$ and $m_t =  \lim_{N \rightarrow \infty} \frac{ \sum_{i=1}^N X^i_t}{N}=\int x \mu_t(dx)$ is the mean position of the population at time $t$}, with $\mu_{0-}$ symmetric around $m_{0-}$.

{
Note that one could write an alternative MFG formulation with  
$$\tilde{v}(\mu_{0-}):=\inf_{(\xi^{+},\xi^{-}) \in  \mathcal{U}_{\infty}} \mathbb{E} \int_0^{\infty} e^{-\alpha t} \left[ h(X_t -m_t)dt + d \check{\xi_t} \right].$$ 
$v(x)$ defined in  \eqref{game} can be viewed as $\tilde{v}(\mu_{0-}|X_{0-}=x)$ with $X_{0-}=x$ as some sample drawn from $\mu_{0-}$. 
Clearly $\tilde{v}(\mu_{0-})$ can be solved by analyzing $v(x)$ as  $\tilde{v}(\mu_{0-}) = \mathbb{E}_{\mu_{0-}}[v(X_{0-})]$. This connection is also explored in  Section 2.2.2 of \cite{LZ2017}. }

The admissible control set for MFG is 
\begin{eqnarray*}
\begin{aligned}
 \mathcal{U}_{\infty}= & \left\{ (\xi_t^+,\xi_t^-) \given  \xi_t^+ \mbox{ and }\xi_t^- \mbox{ are } \mathcal{F}_t^{(X_{t-}, m_{t-})}\mbox{-progressively measurable, c\`adl\`ag, non-decreasing,} \right.\\
&  \hspace{50pt} \left.{}  \mbox{ with }\mathbb{E} \left[ \int_0^{\infty}e^{-\alpha t}d\xi_t^+ \right] <\infty , \enspace  \mathbb{E} \left[ \int_0^{\infty}e^{-\alpha t}d\xi_t^- \right] <\infty, \enspace \xi_{0-}^+=0, \enspace \xi_{0-}^-=0 \right\} .
\end{aligned}
\end{eqnarray*}

\subsection{NE Solution to the MFG}

\begin{definition}[NE to  MFG (\ref{game})]
An NE to the MFG (\ref{game}) is a  pair of Markovian control $(\xi_t^{*,+},\xi_t^{*,-})_{t \geq 0}$ and a mean function $\{m^*_t\}_{ t \ge 0}$ such that
\begin{itemize}
\item ${v}^*(x) =J_{(\infty)}\left(x;\xi^{*,+},\xi^{*,-}|\{\mu^*_t\}_{t\ge0}\right)=\min_{\xi \in \mathcal{U}_{\infty}}J_{(\infty)}\left(x;\xi^{+},\xi^{-}|\{\mu^*_t\}_{t\ge0}\right)$,
\item $P_{X^{*}_{t}}=\mu_t^*$, and $m_t^* = \int x P_{X^{*}_{t}} (dx)$ is the mean function of $X_t^{*}$ where $X_t^{*}$ is the controlled dynamic under $(\xi_t^{*,+},\xi_t^{*,-})_{t \ge 0}$.
\end{itemize}
${v}^*(x)$ is called the NE value of the MFG associated with ${\xi}^*$.
\end{definition}

\begin{theorem}[NE to  MFG (\ref{game})]\label{MFG-SOL}
There exists an NE to the MFG (\ref{game}), \begin{eqnarray}\label{optimal-mfg-control}
\begin{aligned}
\xi^{*,+}_t &=& \max\left\{0,\max_{0 \leq u \leq t}\{ {m_{0-}}-x -B_u + \xi_u^{*,-} -c\} \right\},\\
\xi^{*,-}_t &=& \max \left\{0,\max_{0 \leq u \leq t}\{ x-{m_{0-}} + B_u + \xi_u^{*,+} -c\} \right\},
\end{aligned}
\end{eqnarray}
and the corresponding NE value   is 

\begin{eqnarray}\label{stationary-solution}
    v^*(x)=\left\{
                \begin{array}{ll}
                  -\frac{p_1^{\prime \prime}(c\sqrt{2 \alpha}) \cosh \left(x\sqrt{2 \alpha}\right)}{2 \alpha  \cosh\left(c \sqrt{2 \alpha}\right)} + p_1(x-{m_{0-}}), & {m_{0-}} \leq x \leq {m_{0-}}+ c,\\
                  v({m_{0-}}+c)+(x-{m_{0-}}-c), & x \geq {m_{0-}}+c,\\
                 v({m_{0-}}-x), & x <{m_{0-}},
                \end{array}
              \right.
  \end{eqnarray}
where c is the solution to (\ref{constant_c}).
 
\end{theorem}

The proof consists of  three steps.


\subsubsection*{Step 1: Stochastic control problem.}
Take the $M_1$ topology for the Skorokhod space $\mathcal{D}([0,\infty))$ with a  Wasserstein distance $W_1$(\cite{skorokhod1956, FH2017}). 
Fix a  mean field measure $\{\mu_t\}_{t \geq 0} \in \mathcal{P}_1(\mathcal{D}([0,\infty)))$, with $m_t=\int x \mu_t(dx)$ and $\mathcal{P}_1$  the class of all probability measures with finite moment of first order. Then (\ref{game}) becomes the following time-dependent and state-dependent singular control problem,
 \begin{eqnarray}\label{fixed_m_control}
 \begin{aligned}
\hat{v}(s,x) &= \inf_{{\xi} \in  \mathcal{U}_{\infty}} \mathbb{E}\int_s^{\infty} e^{-\alpha(t-s)}  \left[h(X_t - m_t)dt +d\xi_t^{+} + d\xi_t^{-}\right]\\
\mbox{ such that }& \hspace{20pt} d X_t = d B_t  +d\xi_t^{+} - d\xi_t^{-}, \,X_{s-}=x,\, m_{s-} = m. 
\end{aligned}
\end{eqnarray}

 The corresponding HJB equation for $\hat{v}(s,x)$ is
\begin{eqnarray}\label{hjb_mfg}
\begin{aligned}
&&\max \left\{ \alpha \hat{v} (s,x) - \hat{v} _t(s,x) - \frac{1}{2} \hat{v} _{xx}(s,x) - h(x-m) ,-1+\hat{v}_x(s,x),-1-\hat{v}_x(s,x) \right\} = 0.
\end{aligned}
\end{eqnarray}
Note that (\ref{hjb_mfg}) is a parabolic equation  because of $\mu_t$ despite the infinite horizon. This is different from the elliptic equation (\ref{originalHJB}). 

We will  show that $\hat{v}(s,x)$ in  (\ref{fixed_m_control}) is a viscosity solution to HJB equation (\ref{hjb_mfg}).  

First, under a fixed $\{\mu_t\}_{t \ge 0}$, the following dynamic programming principle holds.
\paragraph{Dynamic programming principle (DPP).} For all $(s,x)\in \mathbb{R}^+ \times \mathbb{R}$, 
\begin{eqnarray}\label{dpp}
\hat{v}(s,x) = \inf_{\xi \in \mathcal{U}_{\infty}} \mathbb{E} \left[ \int_s^{\theta} e^{-\alpha (t-s)} \left(h(X_t-m_t)dt+d\check{\xi}_t\right)+{e}^{-\alpha(\theta-s)}v(\theta,X_{\theta})\right]
\end{eqnarray}
for any $\theta \in \mathcal{T}$ and $\theta \ge s$, with $\mathcal{T}$ the set of all $\{\mathcal{F}^{{(X_t,m_t)}}\}_{t \ge 0}$-stopping times. Here, we adopt 
the convention that $e^{-\alpha \theta(\omega)}=0$ when $\theta(\omega)=\infty$. The proof of DPP (\ref{dpp}) follows  Guo and Pham \cite{GP2005} by extending the state space from $\mathbb{R}$ to $\mathbb{R}^+ \times \mathbb{R}$.

\begin{definition}[Viscosity solution]
$\hat{v}(t,x)$ is a \emph{continuous viscosity solution} to (\ref{hjb_mfg}) on $[0,\infty)\times \mathbb{R} $ if
\begin{itemize}
\item  Viscosity super-solution: for any $(t_0,x_0) \in [0,\infty)\times \mathbb{R}$ and for any function $\phi(t_0,x_0)$ such that $(t_0,x_0)$ is a local minimum of $(\hat{v}-\phi)(t,x)$ with $\hat{v}(t_0,x_0) = \phi(t_0,x_0)$,
\begin{eqnarray*}
\max \left\{  \alpha \phi (t_0,x_0)-\phi_t(t_0,x_0) - \frac{1}{2}\phi_{x,x} (t_0,x_0) - h(x_0-m), -1+\phi_x(t_0,x_0),-1-\phi_x(t_0,x_0) \right\} \geq 0.
\end{eqnarray*}
\item Viscosity sub-solution: for any $(t_0,x_0) \in [0,\infty)\times \mathbb{R}$ and for any function $\phi(t_0,x_0)$ such that $(t_0,x_0)$ is a local maximum of $(\hat{v}-\phi)(t,x)$ with $\hat{v}(t_0,x_0) = \phi(t_0,x_0)$,
\begin{eqnarray*}
\max \left\{ \alpha \phi (t_0,x_0)- \phi_t(t_0,x_0) - \frac{1}{2}\phi_{x,x} (t_0,x_0) - h(x_0-m) , -1+\phi_x(t_0,x_0),-1-\phi_x(t_0,x_0)\right\}  \leq 0.
\end{eqnarray*}
\end{itemize}
\end{definition}

\begin{proposition}\label{viscosity_differentiable}
Assume that the value function $\hat{v}(t,x)$ of  (\ref{fixed_m_control}) is continuous with respect to $t$. Then $\hat{v}(t,x)$  is a continuous viscosity solution of the HJB equation (\ref{hjb_mfg}) on $[s,\infty) \times \mathbb{R}$. Moreover, $\hat{v}(t,x)$ is convex and differentiable in $x$, and for any $x,y \in \mathbb{R}$, 
\begin{eqnarray}\label{property_vt}
\hat{v}(s,x) \leq \hat{v}(s,y) +|x-y| .
\end{eqnarray}
\end{proposition}
\begin{proof}
Since $h$ is convex and the pay-off function $\mathbb{E}\left[\int_s^{\infty} e^{-\alpha(t-s)}h(X_t - m_t)dt +d\xi_t^{+} + d\xi_t^{-}\right]$ in problem (\ref{fixed_m_control}) is linear in control $(\xi^{+},\xi^{-})$, the value function $\hat{v}(s,x)$ is convex in $x$. Since $\hat{v}(s,x)$ is finite and convex on $(-\infty,\infty)$, it is continuous in $x$. 
Moreover, consider a special control,
\begin{eqnarray}
\xi_t^{+}-\xi_t^{-}  =\left\{
                \begin{array}{ll}
                 0,\quad t=s,\\
               y-x,\quad t \geq s,
                \end{array}
              \right.
\end{eqnarray}
clearly $\hat{v}(s,x)\leq \hat{v}(s,y)+|y-x|$ .

We now prove that the value function is a viscosity solution of (\ref{hjb_mfg}).

\begin{itemize}
\item Step A: Viscosity sub-solution.\\
For some $(t_0,x_0) \in \mathbb{R}^+ \times  \mathbb{R}$ and $\phi \in  \mathcal{C}^{1,2}(\mathbb{R}^+ \times  \mathbb{R})$ such that $\hat{v}(t_0,x_0) = \phi(t_0,x_0)$ and $\phi(t_0,x_0) \geq \hat{v}(t_0,x_0)$ for $(t,x) \in B_{\epsilon}(t_0,x_0)$. That is, $\hat{v}-\phi$ has local maximum at $(t_0,x_0)$. Consider the following admissible control
\begin{eqnarray}
\xi_t^+  =\left\{
                \begin{array}{ll}
                 0,\quad t =t_0,\\
               \eta_1,\quad t \geq t_0,
                \end{array}
              \right.
\end{eqnarray}
\begin{eqnarray}
\xi_t^- =\left\{
                \begin{array}{ll}
                 0,\quad t=t_0,\\
               \eta_2,\quad t \geq t_0,
                \end{array}
              \right.
\end{eqnarray}
where $0 \leq \eta_1,\eta_2 \leq \epsilon$. Define the exit time 
\begin{eqnarray}
\tau_{\epsilon} = \inf \left\{t \geq t_0, X_t \notin \bar{B}_{\epsilon}(t_0,x_0)\right\}.
\end{eqnarray}
Notice that $X$ has at most one jump at $t = t_0$ and is continuous on $[t_0,t_0+\tau_{\epsilon})$. By the DPP,
\begin{eqnarray}\label{65}
\begin{aligned}
\phi(t_0,x_0) =\hat{v}(t_0,x_0) &\leq& \mathbb{E}  \int_{t_0}^{t_0+\tau_{\epsilon}\wedge \delta} e^{-\alpha(t-t_0)}\left[h(X_t-m_t)dt + d\xi_t^{+}+ d\xi_t^{-}\right] \\ 
&\enspace &+ \enspace \mathbb{E}\left[e^{-\alpha(\tau_{\epsilon}\wedge \delta)}\phi(t_0+\tau_{\epsilon}\wedge \delta,X_{t_0+\tau_{\epsilon}\wedge \delta})\right].
\end{aligned}
\end{eqnarray}
By It\^{o}'s lemma,
\begin{eqnarray}\label{66}
\begin{aligned}
 & \mathbb{E}[e^{-\alpha(\tau_{\epsilon}\wedge \delta)}\phi(t_0+\tau_{\epsilon}\wedge \delta,X_{t_0+\tau_{\epsilon}\wedge \delta})]\\
  = \phi(t_0,x_0) &+ \mathbb{E} \left[\int_{t_0}^{t_0+\tau_{\epsilon}\wedge \delta} e^{-\alpha(t-t_0)}(-\alpha \phi + \phi_t + \frac{1}{2} \phi_{x,x})(t,X_t)dt \right]\\
  \quad& + \enspace \mathbb{E} \left[\sum_{t_0 \leq t \leq \tau_{\epsilon}\wedge \delta} e^{-\alpha t }(\phi(t,X_t)-\phi(t,X_{t-})) \right].
\end{aligned}
\end{eqnarray}
Combining (\ref{65}) and (\ref{66}),
\begin{eqnarray}
\begin{aligned}
& \mathbb{E} \left[\int_{t_0}^{t_0+\tau_{\epsilon}\wedge \delta} e^{-\alpha(t-t_0)}(\alpha \phi - \phi_t -\frac{1}{2}\phi_{x,x}-h)(t,X_t)dt \right] \\
-&  \mathbb{E} \left[\int_{t_0}^{t_0+\tau_{\epsilon}\wedge \delta} e^{-\alpha(t-t_0)}(d\xi_t^{+}+d\xi_t^{-}) \right]\\
-&  \mathbb{E}\left[\sum_{t_0 \leq t \leq \tau_{\epsilon}\wedge \delta} e^{-\alpha t }(\phi(t,X_t)-\phi(t,X_{t-})) \right] \leq 0.
\end{aligned}
\end{eqnarray}
Now,  setting $\eta_1=\eta_2=0$ and letting $\delta \rightarrow 0$ leads to $\alpha \phi -\phi_t -\frac{1}{2}\phi_{x,x}-h \leq 0$.

Next, let $\eta_2=0$, and note that $\xi^+_t$ and $X_t$ only jump at time $t_0$ with a size $\eta_1$, therefore
\begin{eqnarray*}
 \mathbb{E} \left[\int_{t_0}^{t_0+\tau_{\epsilon}\wedge \delta} e^{-\alpha(t-t_0)}(\alpha \phi -\phi_t -\frac{1}{2}\phi_{x,x}-h)(t,X_t)dt \right] - \eta_1- \phi(t_0,x_0+\eta_1) + \phi(t_0,x_0)\leq 0.
\end{eqnarray*}
Now, taking $\delta \rightarrow 0$, dividing by $\eta_1$, and letting $\eta_1 \rightarrow 0$ yields $-1-\phi_x \leq 0$.
Similarly,  $-1+ \phi_x \leq 0$. That is, $\phi$ is the sub-solution to (\ref{hjb_mfg}), so that 
\begin{eqnarray*}
\max \left\{\alpha \phi  (t_0,x_0)- \phi_t  (t_0,x_0)-\frac{1}{2}\phi_{x,x} (t_0,x_0)-h(x_0-m),-1-\phi_x (t_0,x_0),-1+ \phi_x (t_0,x_0) \right\} \leq 0.
\end{eqnarray*}

\item Step B: Viscosity Super-solution.\\
This is established by a contradiction argument. Suppose otherwise, then there exists $(t_0,x_0)$, $\epsilon,\delta>0$ $\phi \in C^{1,2}(\mathbb{R}^+ \times \mathbb{R})$ such that for any $(t,x) \in \bar{B}_{\epsilon}(t_0,x_0)$,
\begin{eqnarray}
\begin{cases}
\alpha \phi - \frac{1}{2}\phi_{x,x} - h(x-m) -\phi_t\leq -\delta,\\
-1+\delta \leq  \phi_x \leq 1-\delta.
\end{cases}
\end{eqnarray}
Given any admissible control $(\xi^+,\xi^-)\in  \mathcal{U_{\infty}}$, consider an exit time $\tau_{\epsilon} = \inf \{t \geq 0, X_{t+t_0} \notin \bar{B}_{\epsilon}(t_0,x_0)\}$, and apply It\^{o}'s lemma to $e^{-\alpha t}\phi(t,X_t)$, 
\begin{eqnarray*}
\begin{aligned}
 \mathbb{E} \left[ e^{-\alpha \tau_{\epsilon} }\phi(t_0+\tau_{\epsilon},X_{t_0+\tau_{\epsilon}}) \right] &= \phi(t_0,x_0)+ \mathbb{E} \left[\int_{t_0}^{t_0+\tau_{\epsilon}} e^{-\alpha(t-t_0)}(-\alpha \phi +\phi_t +\frac{1}{2}\phi_{x,x})(t,X_t)dt \right]\\
  &\enspace\enspace+ \enspace \mathbb{E} \left[\sum_{t_0 \leq t \leq \tau_{\epsilon}} e^{-\alpha t }(\phi(t,X_t)-\phi(t,X_{t-})) \right]\\
   &\enspace\enspace+ \enspace \mathbb{E} \left[
   \int_{t_0}^{t_0+\tau_{\epsilon}} e^{-\alpha t} \phi^{\prime}(t,X_t) \left(\left(d \xi_t^{+}\right)^c+\left(d \xi_t^{-}\right)^c\right)
   \right].
\end{aligned}
\end{eqnarray*}
Notice that for any $t_0 \leq t \leq t_0 + \tau_{\epsilon}$, $(t,X_t) \in \bar{B}_{\epsilon}(t_0,x_0)$. By the Taylor expansion and $\Delta X_t = \Delta \xi^+_t-\Delta \xi^-_t$, clearly for any $0 \leq t < \tau_{\epsilon}$:
\begin{eqnarray}\label{611}
\phi(t,X_t) - \phi(t,X_{t-}) &=& \Delta X_t \int _0^1 \phi_x(t,X_t+z \Delta X_t)dz \nonumber\\
&\geq&  (-1+\delta) (\Delta \xi_t^+ + \Delta \xi_t^-).
\end{eqnarray}
Thus, 
\begin{eqnarray}\label{612}
\begin{aligned}
&\mathbb{E}[e^{-\alpha \tau_{\epsilon} }\phi(t_0+\tau_{\epsilon},X_{t_0+\tau_{\epsilon}-})] \\
 \geq 
\phi(t_0,x_0) &+ \mathbb{E}\left[\int_{t_0}^{t_0+\tau_{\epsilon}}e^{-\alpha (t-t_0)}(-h + \delta)(t,X_t)dt \right]\\
\quad &+(\delta-1) \mathbb{E} \left[\int_{t_0}^{t_0+\tau_{\epsilon}-}e^{-\alpha (t-t_0)} (d\xi_t^+ +d\xi_t^- ) \right]\\
= \phi(t_0,x_0) &+ \mathbb{E} \left[\int_{t_0}^{t_0+\tau_{\epsilon}}e^{-\alpha (t-t_0)}\left(-h(X_t-m_t)dt-d\xi_t^+-d\xi_t^- \right)\right]\\
\quad &+\mathbb{E} \left[e^{-\alpha \tau_{\epsilon}}(\Delta \xi_{t_0+\tau_{\epsilon}}^++\Delta \xi_{t_0+\tau_{\epsilon}}^-)] +\delta \mathbb{E}[\int_{t_0}^{t_0+\tau_{\epsilon}}e^{-\alpha t}dt \right] \\
\quad &+ \delta \mathbb{E} \left[\int_{t_0}^{t_0+\tau_{\epsilon}-}e^{-\alpha (t-t_0)} (d\xi_t^+ +d\xi_t^-) \right].
\end{aligned}
\end{eqnarray}
By definition of $\tau_{\epsilon}$, $(t_0+\tau_{\epsilon}-,X_{t_0+\tau_{\epsilon}-}) \in \bar{B}_{\epsilon}(t_0,x_0)$ and $(t_0+\tau_{\epsilon},X_{t_0+\tau_{\epsilon}})$ is either on the boundary $\partial {B}_{\epsilon}(t_0,x_0)$ or out of $\bar{B}_{\epsilon}(t_0,x_0)$. However, there exists some random variable $\alpha \in [0,1]$ such that,
\begin{eqnarray*}
x_{\alpha} &=& X_{t_0+\tau_{\epsilon}-}+ \alpha \Delta X_{t_0+\tau_{\epsilon}}\\
 &=& X_{t_0+\tau_{\epsilon}-}+ \alpha (\Delta \xi^+_{t_0+\tau_{\epsilon}}-\Delta \xi^-_{t_0+\tau_{\epsilon}}) \in \partial B_{\epsilon}(t_0,x_0).
\end{eqnarray*}
Similar as in (\ref{611}), we have
\begin{eqnarray}\label{613}
\begin{aligned}
\phi(t_0+\tau_{\epsilon},x_{\alpha}) - \phi(t_0+\tau_{\epsilon},X_{t_0+\tau_{\epsilon}-}) 
\geq  \alpha (-1+\delta) (\Delta \xi_{t_0+\tau_{\epsilon}}^+ + \Delta \xi_{t_0+\tau_{\epsilon}}^-).
\end{aligned}
\end{eqnarray}
Notice that $X_{t_0+\tau_{\epsilon}} = x_{\alpha}+(1-\alpha)( \Delta \xi_{t_0+\tau_{\epsilon}}^+ - \Delta \xi_{t_0+\tau_{\epsilon}}^-)$, and from (\ref{property_vt}), 
\begin{eqnarray}\label{614}
\hat{v}(t_0+\tau_{\epsilon},x_{\alpha}) \leq \hat{v}(t_0+\tau_{\epsilon},X_{t_0+\tau_{\epsilon}}) + (1-\alpha)(\Delta \xi_{t_0+\tau_{\epsilon}}^+ + \Delta \xi_{t_0+\tau_{\epsilon}}^-).
\end{eqnarray}
Recalling $\phi(t_0+\tau_{\epsilon},x_{\alpha}) \leq \hat{v}(t_0+\tau_{\epsilon},x_{\alpha})$, inequalities (\ref{613}) and (\ref{614}) imply 
\begin{eqnarray*}
   \phi(t_0+\tau_{\epsilon},X_{t_0+\tau_{\epsilon}-}) \leq \hat{v}(t_0+\tau_{\epsilon},X_{t_0+\tau_{\epsilon}}) + (1-\alpha \delta)(\Delta \xi_{t_0+\tau_{\epsilon}}^+ + \Delta \xi_{t_0+\tau_{\epsilon}}^-) .
\end{eqnarray*}
Plugging the above inequality into (\ref{612}), by $\phi(t_0,x_0) = \hat{v}(t_0,x_0)$, 
\begin{eqnarray}\label{615}
\begin{aligned}
\mathbb{E}e^{-\alpha \tau_{\epsilon} }& \left[\int_{t_0}^{t_0+\tau_{\epsilon}}\left(h(X_t-m_t)dt +d\xi_t^++d\xi_t^-\right)  + \hat{v}(t_0+\tau_{\epsilon},X_{t_0+\tau_{\epsilon}}) \right] \\
\geq \hat{v}(t_0,x_0) &+  \alpha \delta \mathbb{E} \left[ e^{-\alpha \tau_{\epsilon}}(\Delta \xi_{t_0+\tau_{\epsilon}}^++\Delta \xi_{t_0+\tau_{\epsilon}}^-) \right] \\
&+\delta \mathbb{E} \left[\int_{t_0}^{t_0+\tau_{\epsilon}}e^{-\alpha t}dt \right] + \delta \mathbb{E} \left[\int_{t_0}^{t_0+\tau_{\epsilon}-}e^{-\alpha (t-t_0)} (d\xi_t^+ +d\xi_t^-) \right].
\end{aligned}
\end{eqnarray}
There exists a constant $g_0>0$ such that for any $(\xi^+,\xi^-)\in  \mathcal{U}_{\infty}$,
\begin{eqnarray*}
\alpha \mathbb{E} \left[e^{-\alpha \tau_{\epsilon}}(\Delta \xi_{t_0+\tau_{\epsilon}}^+ +\Delta \xi_{t_0+\tau_{\epsilon}}^-)\right] + \mathbb{E} \left[\int_{t_0}^{t_0+\tau_{\epsilon}}e^{-\alpha t}dt \right] 
+  \mathbb{E} \left[\int_{t_0}^{t_0+\tau_{\epsilon}-}e^{-\alpha (t-t_0)} (d\xi_t^+ +d\xi_t^-)\right] \geq g_0.
\end{eqnarray*}
Finally,  taking the  infimum over all admissible controls $(\xi^+,\xi^-) \in  \mathcal{U}_{\infty}$ in (\ref{615}) suggests
\begin{eqnarray}
\hat{v}(t_0,x_0) \geq \hat{v}(t_0,x_0) + \delta g_0,
\end{eqnarray}
which is a contradiction.
\end{itemize}

The differentiability with respect to $x$ can be proved using the convexity  of the value function $\hat{v}(s,x)$ to (\ref{hjb_mfg}). Since $\hat{v}(s,x)$ is convex, the left and right derivatives with respect to $x$, $\hat{v}_{x-}(t,x)$ and $\hat{v}_{x+}(t,x)$ exist for any $t \geq s$ and $x \in \mathbb{R}$. Also, $\hat{v}_{x-}(t,x) \leq \hat{v}_{x+}(t,x)$ by convexity. We argue by contradiction and suppose there exists $x_0 \in \mathbb{R}$ and $t_0 \geq 0$ such that $\hat{v}_{x-}(t_0,x_0) < \hat{v}_{x+}(t_0,x_0)$. Fix some $q$ in $(\hat{v}_{x-}(t_0,x_0) , \hat{v}_{x+}(t_0,x_0))$ and consider the test function
  \begin{eqnarray*}
  \phi_{\epsilon}(t,x) = \hat{v}(t_0,x_0) +q(x-x_0) -\frac{1}{2\epsilon}(x-x_0)^2 -\frac{1}{2\epsilon}(t-t_0)^2,
  \end{eqnarray*}
  with $\epsilon>0$. Then $(t_0,x_0)$ is a local minimum of $(\hat{v}-\phi_{\epsilon})(t,x)$ since $\hat{v}_{x-}(t_0,x_0) < q=\phi_x(t_0,x_0)< \hat{v}_{x+}(t_0,x_0) $ and $\phi_t (t_0,x_0)=0$. Hence $\phi$ is a viscosity super-solution by definition. That is,
  \begin{eqnarray*}
\max \left\{\alpha \phi - \phi_t -\frac{1}{2}\phi_{x,x}-h(x_0-m),-1-\phi_x,-1+ \phi_x \right\} \geq 0,
\end{eqnarray*}
  which leads to $ -\frac{1}{2 \epsilon}+h(x_0-m)-\alpha \phi(t_0,x_0) \geq 0$. Taking $\epsilon>0$ sufficiently small  leads to a contradiction.
\end{proof}

\begin{proposition}[Optimal Control] \label{optimal_control_fixed_mu}
Assume \textbf{A1} and assume that $\hat{v}_t(t,x)$ is continuous with respect to $t$, the optimal control  to (\ref{fixed_m_control}) under a fixed $\{\mu_t\}_{t \geq 0}\in \mathcal{P}_1(\mathcal{D}([0,\infty)))$ is of the form
\begin{eqnarray}\label{optimal-control-fixed-mu_t}
 d \hat{\xi}_t=\left\{
                \begin{array}{ll}
                  m_t+c_t-x, \quad & \hat{v}_x(t,x) = 1,\\
                  0, & \vert \hat{v}_x(t,x) \vert< 1,\\
                  m_t-c_t-x, &  \hat{v}_x(t,x) = -1,
                \end{array}
              \right.
\end{eqnarray}
where $t \geq 0$, $m_t = \int x \mu_t (dx)$, and $c_t=\inf\{x \given \hat{v}_x(t,x) = 1\} -m_t =-\sup\{x \given  \hat{v}_x(t,x) = -1\} + m_t$. \end{proposition}
\begin{proof}
By Proposition \ref{viscosity_differentiable}, $\hat{v}(t,x)$ is convex and differentiable in $x$, hence for any fixed $t \in [0, \infty)$, $c^1_t:=\inf\{x \given \hat{v}_x(t,x) = 1\} -m_t$ and $c^2_t:=-\sup\{x \given \hat{v}_x(t,x) = -1\} + m_t$ exist. By the symmetry of Problem (\ref{game}) under a fixed $\{m_t\}_{t \geq 0}$, $\hat{v}(t,m_t+\delta) = \hat{v}(t,m_t-\delta)$ and $\hat{v}_x(t,m_t+\delta) = -\hat{v}_x(t,m_t-\delta)$ for any fixed $t$ and any $\delta >0$, hence  $c_t^1 = c_t^2$,  denoted as $c_t$.

Because $\hat{v}(t,x)$ is convex in $x$ and continuously differentiable in $x$ and $t$,  one can apply the generalized It\^{o}'s formula to $\hat{v}(t,x)$ with (\ref{optimal-control-fixed-mu_t}) and use a similar argument as the verification theorem in \cite{Karatzas1982} to obtain the optimality of (\ref{optimal-control-fixed-mu_t}).
\end{proof}

Given the optimal control (\ref{optimal-control-fixed-mu_t}), define a mapping  $\Gamma_1: \mathcal{P}_1(\mathcal{D}([0,\infty)))\rightarrow\mathcal{D}([0,\infty))$ such that 
$$\Gamma_1 \left(\{\mu_t\}_{t \geq 0} \right) = \{ \hat{\xi}\ \vert \ \{\mu_t\}_{t \geq 0}\}_{t \geq 0}.$$

\subsubsection*{Step 2: Consistency.}
Given Proposition \ref{optimal_control_fixed_mu} and a fixed flow $\{\mu_t\}_{t\geq 0}$, the optimal control $(\hat{\xi}_t^{+},\hat{\xi}_t^{-})$ to  (\ref{hjb_mfg}) is a bang-bang type and the controlled process $\hat{{X}}_t$ is a reflected Brownian motion with two time-dependent reflected boundaries $m_t+c_t$ and $m_t-c_t$. 
$m_t+c_t,m_t-c_t\in \mathcal{C}([0,\infty])$ since $\hat{v}(t,x)$ is continuous and differentiable. By Theorem 2.6 in Burdzy, Kang, and Ramanan~\cite{BKR2009}, there exists a unique solution, $\hat{{X}}_t$, to the SP with time varying domain $ \{(t,x) \given m_t-c_t \leq x\leq c_t+m_t\}$ such that $\hat{{X}}_t$ is a c\`ad\`ag process. 
Furthermore, by Theorem {2.9} in Burdy, Chen, and Sylvester~\cite{BCS2004}, the Kolmogorov forward equation for $\hat{\mu}_t$ can be described as

\begin{eqnarray}
\left\{
\begin{array}{ll}
p_t(t,x) - \frac{1}{2} p_{x,x}(t,x)=0, &\quad \mbox{ when }  \vert x-m_t \vert < c_t, \label{kol1}\\
p_x(t,x)+2 (\frac{\partial{{m}_t}}{\partial{t}}+\frac{\partial{{c}_t}}{\partial{t}})p(t,x) = 0, &\quad \mbox{ when } x=m_t+c_t,\\
p_x(t,x)-2 (\frac{\partial{{m}_t}}{\partial{t}}-\frac{\partial{{c}_t}}{\partial{t}})p(t,x) = 0, &\quad \mbox{ when } x=m_t-c_t, 
  \end{array}
   \right.
\end{eqnarray}
with the initial distribution $p(0,x)=\hat{\mu}_{0} \in \mathcal{P}_1(\mathbb{R})$, where 
\begin{eqnarray}
\hat{\mu}_0(x)=\left\{
\begin{array}{ll}
0, \,\,\,\,\,\,\,\,& x<m_{0-}-c_0 \mbox{ or } x>m_{0-}+c_0,\\
{\mu}_{0-}(x) , \,\,\,\, \,\,\,\, & \vert x-m_{0-}\vert <c_0,\\
{\mu}_{0-}(x)+ \int_{-\infty}^{m_{0-}-c_0-}{\mu}_{0-}(dx), \,\,\,\, \,\,\,\, &x=m_{0-}-c_0,\\
{\mu}_{0-}(x)+ \int_{m_{0-}+c_0+}^{\infty}{\mu}_{0-}(dx), \,\,\,\, \,\,\,\, &x=m_{0-}+c_0.
\end{array}
\right.
\end{eqnarray}
 By Theorem 2.9 in \cite{BCS2004}, given $m_t+c_t,m_t-c_t \in \mathcal{C}([0,\infty))$, the Kolmogorov forward equation (\ref{kol1}) with the initial distribution {$p(0,x):=\hat{\mu_0}(x)$} has a solution.

\paragraph{Step 3: Fixed point analysis.}
Denote $\hat{\mu}_t$ as the distribution of $\hat{X}_t$, obviously $\hat{\mu}_t \in \mathcal{P}_1(\mathcal{D}([0,\infty))$. Consequently, define $\Gamma_2:\mathcal{D}([0,\infty)) \rightarrow \mathcal{P}_1(\mathcal{D}([0,\infty))$ such that
\[
\Gamma_2\left(\hat{\xi}(t,x|\{\mu_t\}_{t\ge 0}) \right) = \{\hat{\mu}_t\}_{t\ge 0}.
\]


Now, define a mapping $\Gamma: \mathcal{P}_1(\mathcal{D}([0,\infty)) \rightarrow \mathcal{P}_1(\mathcal{D}([0,\infty))$ such that $$
\Gamma(\{\mu_t\}_{t \geq 0}) = \Gamma_2 \circ\Gamma_1(\{\mu_t\}_{t\ge 0}) = \{\hat{\mu}\}_{t\geq 0}.
$$
One can then update $m^{\prime}_t$, and have 
\begin{eqnarray}
d m^{\prime}_t &=& d \left( \int x p(t,dx) \right) \label{k1}\\
&=& \left[ \frac{1}{2}\int x p_{x,x}(t,dx)\right] dt \label{k2}\\
&=&\frac{1}{2} [ x p_x(t,x)\vert_{x=m_t+c_t}- x p_x(t,x)\vert_{x=m_t-c_t}-p(t,x)\vert_{x=m_t+c_t}+p(t,x)\vert_{x=m_t-c_t}]dt\label{k3}\\
&=& \frac{1}{2} \left.\left[\left( -2 \left(\frac{d{m}_t}{d t}+\frac{d{c}_t}{dt}\right)x-1 \right)p(t,x) \right\vert_{x=m_t+c_t}\right.\nonumber\\ 
&&\left.\left.-\left(2\left(\frac{d{m}_t}{dt}-\frac{d{c}_t}{dt}\right)x-1\right)p(t,x) \right\vert_{x=m_t-c_t}\right] dt\label{k4}
\end{eqnarray}
(\ref{k2}) comes from  (\ref{kol1}), (\ref{k3}) is from integration by part, and (\ref{k4}) follows from the boundary conditions. Since {$\mu_{0-}$} is symmetric around $m_{0-}$ and the optimal control (\ref{optimal-control-fixed-mu_t}) is an odd function around $m_t$ for any $t \geq 0$, the distribution $p(t,x)$ is symmetric around $m_{0-}$ for any $t \geq 0$.

\begin{eqnarray}\label{fixed_point}
(\ref{k4})= -2\left(\frac{d{m}_t}{dt} m_t+\frac{d{c}_t}{dt} c_t \right)p(t,m_t+c_t) dt.\label{k5}
\end{eqnarray}
Clearly $m_t= m_{0-}$ is one solution to the fixed point equation (\ref{fixed_point}).
This fixed point to $\Gamma$  is an NE to the MFG (\ref{game}) and the associated NE value is smooth in both $x, t$.

{
\begin{remark} \label{remark4} Note that solution  $m_t \ (=m_{0-})$ is time independent and distribution independent. 
Consequently $v(t,x)$ is time independent and  $\frac{d c_t}{dt}=0$. In fact, this time independent property of the value function $v(t,x)$ reduces the HJB equation (\ref{hjb_mfg}) from a parabolic form to an elliptic one.
 However, there might be time-dependent NE solution(s) with non-constant mean position $\{m_t\}_{t \ge 0}$ for 
 Eqn. (\ref{fixed_point}). We are unable  to verify the existence/nonexistence of such solutions. 
 
 On a related note,  if instead a stationary MFG (SMFG) is specified by replacing $h(X_t-m_t)$ with
$h(X_t-\lim_{t \rightarrow \infty}m_t)$, the associated  HJB equation (\ref{hjb_mfg}) will also be elliptic.
(See Appendix D for more precise definition of the SMFG formulation.)
 In this case, one can use the same approach to derive  infinitely many NEs of the bang-bang type, with the controlled dynamics reflected 
 at $m-c$ and $m+c$ for any constant $m$. Note however, the NE for the SMFG when $m\ne m_{0-}$ is not an NE for the 
MFG (\ref{game}). 
\end{remark}}


\section{Relation between the $N$-player game and the MFG}

\subsection{Convergence of Game Values}
First,  from Theorem \ref{NE_point_N}, one can see, with the detailed proof given  in Appendix C, 
\begin{proposition} \label{constant_convergence}
Given $c_N$ the unique solution to (\ref{c2}) and $c>0$  the unique solution to (\ref{constant_c}),
$$
\lim_{N \rightarrow \infty}c_N = c.$$
When $h(x)=x^2$,  $c_N$ is a decreasing function of $N$.
 \end{proposition}
{
\begin{remark}\label{remark3}
It is no surprise from our earlier analysis that   MFGs  are  different in nature from  $N$-player games. 
For instance, the MFG  degenerates to a single-player game in the sense that its NE 
is threshold-type bang-bang policy where the threshold is state independent
 while the NEs for the $N$-player game are state dependent. Nevertheless, it is 
still somewhat unexpected to see the total collapse  of the MFG to the single player problem from the above proposition.
This could be a result of  {\it over aggregation} in the MFG formulation: players become more {\it anticipative} when they are assumed to be identical.
\end{remark}
}

Next, denote $v_{(N)}^i$ as the NE value of player $i$ in the $N$-player game. By (\ref{n-value-i}), when $x_1=\cdots=x_N=x$, 
\begin{eqnarray}\label{value_same_initial}
v_{(N)}^i(x,x,\cdots,x) = \frac{-p^{\prime \prime}_N(c_N)}{\frac{2(N-1) \alpha}{N}  \cosh \left(c_N\sqrt{\frac{2(N-1)\alpha}{N}}\right)} +p_N(0).
\end{eqnarray}
In particular,  $v_{(N)}^i(x,x,\cdots,x)$ is independent of $x$.  Moreover, from  Proposition \ref{constant_convergence} and the smoothness of $P_N(x)$,  it is easy to verify that
\begin{eqnarray*}
\left\{
\begin{array}{ll}
&\lim_{N \rightarrow \infty} p^{\prime \prime}_N(c_N) = p^{\prime \prime}_N(c),\\
&\lim_{N \rightarrow \infty} \frac{1}{\frac{2(N-1) \alpha}{N}  \cosh \left(c_N\sqrt{\frac{2(N-1)\alpha}{N}}\right)} = \frac{1}{2 \alpha \cosh \left(c\sqrt{2\alpha}\right)},\\
&\lim_{N \rightarrow \infty} p_N(0) = p_1(0).
\end{array} \right.
\end{eqnarray*}
That is, 
\begin{proposition}\label{proposition-convergence}
For any $x \in \mathbb{R}$,  
$\lim_{N \rightarrow \infty} v_{(N)}^i(x,x,\cdots,x) = v^{*}(x)$, where $v^{*}$ is the NE value of player $i$ in  MFG (\ref{game}) with $\mu_{0-}=\delta(x)$.
\end{proposition}


Figure \ref{fig:v_N} shows the convergence of $v_{(N)}^i(x,x,\cdots,x)$ 
with $h=x^2$ and with different choices of $\alpha$. The MFG is illustrated by the dashed red horizontal line. 
%
%

\begin{figure}[H] 
\centering
\begin{subfigure}{.3\textwidth}
  \centering
  \includegraphics[width=1.0\linewidth]{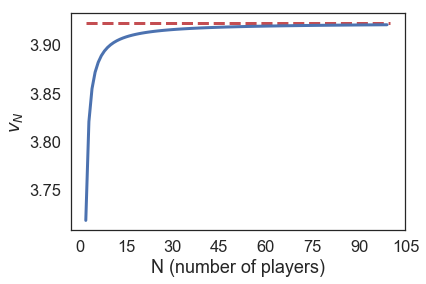}
  \caption{$\alpha=0.2$}
  \label{fig:v_N_1}
\end{subfigure}
\begin{subfigure}{.3\textwidth}
  \centering
  \includegraphics[width=1.0\linewidth]{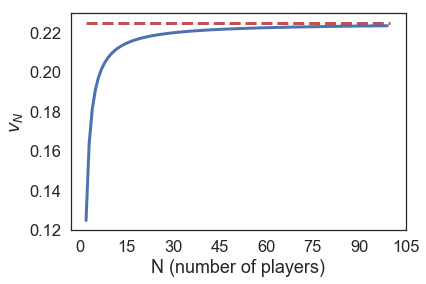}
  \caption{$\alpha=2$}
  \label{fig:v_N_2}
\end{subfigure}%
\begin{subfigure}{.3\textwidth}
  \centering
  \includegraphics[width=1.0\linewidth]{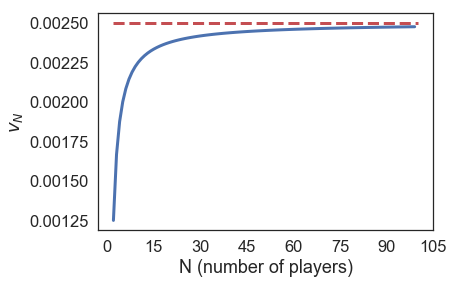}
  \caption{$\alpha=20$}
  \label{fig:v_N_3}
\end{subfigure}
\caption{\label{fig:v_N}
             Convergence of $v_N$ with different discount factors}
\end{figure}

{
\begin{remark}\label{remark2}
 Figure~\ref{fig:v_N} indicates that $v_N$ is an increasing function of $N$ given any fixed decay parameter $\alpha$. This implies that when the number of players increases, it is more costly for players to keep track of other players before making  decisions. Meanwhile,  $v^*(x)$ being a decreasing function of $\alpha$ indicates that the bigger the $\alpha$, the less frequent players will  intervene.
 
 \end{remark}
 }

\subsection{Approximating the $N$-player Game by the MFG}

One can further show that the NE of MFG given in (\ref{optimal-mfg-control}) is an  $\epsilon$-NE for the game in (\ref{N-game}).  

\begin{definition}[$\epsilon$-NE] For the game (\ref{N-game}) with an initial distribution $\mu_{0-}$, a control vector $\pmb{\xi} =({\xi}^1,\ldots,{\xi}^N)$ is called 
its \emph{$\epsilon$-NE}, if for any $i=1,\ldots,N$ and any control ${\xi}^{i'}$ such that $\left(\pmb{\xi}^{-i},{\xi}^{i '}\right) = \left({\xi}^{1},\ldots,{\xi}^{i-1},{\xi}^{i'},{\xi}^{i+1},\ldots,{\xi}^{N}\right)\in \mathcal{S}_N$, 
\begin{eqnarray}\label{epsilon_ne}
\mathbb{E} \left[J^i_{(N)} \left(\pmb{X}_{0-};\pmb{\xi}\right) \right]\leq \mathbb{E} \left[J^i _{(N)}\left(\pmb{X}_{0-};\left(\pmb{\xi}^{-i},{\xi}^{i '}\right)\right) \right]+ \epsilon.
\end{eqnarray}
Here $X_{0-}^{i}$ $(i=1,2,\cdots,N)$ are independent samples from distribution $\mu_{0-}$, and $\mathcal{S}_N$ is defined in (\ref{S_N}).
\end{definition}

\begin{theorem}[$\epsilon$-NE of the $N$-player game]\label{MFG_approximation}
Let  $\xi^*$ be  the NE of MFG  given in (\ref{optimal-mfg-control}), then it is an $\epsilon$-NE of the game (\ref{N-game}), with $\epsilon = O\left(\frac{1}{\sqrt{N}}\right)$.
\end{theorem}

\begin{proof}

Given the game (\ref{N-game}) with  $m_{0-}=\int x \mu_{0-}(dx) \in \mathbb{R}$, assume that each player $i$ in the $N$-player game takes the control $(\xi^{i*,+}_t,\xi^{i*,-}_t)$ according to  the NE  of the MFG such that 
\begin{eqnarray}\label{optimal-mfg-n-player-game}
\begin{aligned}
\xi^{i*,+}_t &=& \max\left\{0,\max_{0 \leq u \leq t}\left\{ m_{0-}-X^i_{0-} -B_u^i + \xi_u^{i*,-} -c\right\}\right \},\\
\xi^{i*,-}_t &=& \max \left\{0,\max_{0 \leq u \leq t}\left\{ X^i_{0-}-m_{0-}+ B_u^i + \xi_u^{i*,+} -c\right\} \right\}.
\end{aligned}
\end{eqnarray}
To see that  $\pmb{\xi}^{*}= \left({\xi}^{1*},\ldots,{\xi}^{N*} \right) \in \mathcal{S}_N$, 
 define 
\begin{eqnarray*}
\mathcal{CW}_{mfg} &=& \left\{\left.\pmb{x} \in \mathbb{R}^N\right \vert \ |x_i-m_{0-}| <c \mbox{ for } i=1,2,\cdots,N  \right\},\\
E_{mfg,i}^{-} &=& \left\{ \left.\pmb{x} \in \mathbb{R}^N \right\vert \enspace x^i-m_{0-} \leq -c \right\},\nonumber\\
E_{mfg,i}^{+} &=& \left\{\left. \pmb{x} \in \mathbb{R}^N \right \vert \enspace x^i-m_{0-}  \geq c \right\},\nonumber
\end{eqnarray*}
with the partition
\begin{eqnarray*}
Q_{{mfg},i} &=& \left\{ \left.\pmb{x} \in \mathbb{R}^N  \right\vert \quad  \left|x^i - m_{0-}\right|  \geq  \left|x^k- m_{0-} \right|, \mbox{ for any } k < i ; \right.\nonumber\\ 
&&\hspace{56pt}\left.{}  \left|x^i -m_{0-}\right|  >  \left|x^k-m_{0-} \right| , \mbox{ for any } k > i  \right\}.\nonumber
\end{eqnarray*}
Then the control in (\ref{optimal-mfg-n-player-game}) corresponds to the  action region 
$ \mathcal{A}_{mfg,i} = \left\{E_{mfg,i}^{-} \cup E_{mfg,i}^{+}\right\} \cap Q_{mfg,i}.$
The independence of $\{B_t^1,\ldots,B_t^N\}$ and the continuity of $\{X_t^{1*},\ldots,X_t^{N*}\}_{t>0}$  imply that for any $t \ge 0$
$\mathbb{P} \left(\Pi_{i=1,\ldots,N}d \xi^{i*}_t=0 \right)=1$.

Suppose that only one player, and without loss of generality, player one, deviates her control $\eta_t = (\eta_t^+,\eta_t^-)$ from all the other players  such that $(\pmb{\xi}^{-i*},{\eta}) \in \mathcal{S}_N$. 
Let $\hat{X}_t^1$ be the new position of player one under control $(\eta_t^+,\eta_t^-)$ with initial value $X_{0-}^1$. 
 Then

\begin{eqnarray*}
&& h \left( \hat{X}^1_t - \frac{1}{N} \left( \hat{X}_{t}^1 +\sum_{j=2, \ldots, N}X_{t}^j \right) \right) \\
&=& h \left( \hat{X}^1_t - \frac{1}{N} \hat{X}_{t}^1 - \frac{N-1}{N} m_{0-} \right)  + h^{\prime} \left(\hat{X}^1_t - \frac{1}{N} \hat{X}_{t}^1 - \frac{N-1}{N} m_{0-} \right)\left(\frac{\sum_{j=2, \ldots, N}X_{t}^j}{N} -\frac{N-1}{N}m_{0-} \right)\\
&\quad& + \frac{h^{\prime \prime}(U_t)}{2}\left(\frac{\sum_{j=2, \ldots, N}X_{t}^j}{N} -\frac{N-1}{N}m_{0-} \right)^2,
\end{eqnarray*}
where $U_t$ is a process  between  $\left(\hat{X}^1_t - \frac{1}{N} \hat{X}_{t}^1 - \frac{N-1}{N} m_{0-}\right) $ and $\left(\hat{X}^1_t - \frac{1}{N} \left( \hat{X}_{t}^1 +\sum_{j=2}^{ N}X_{t}^j\right)\right)$.
By  Assumption \textbf{A1},  
\begin{eqnarray*}
&& h \left( \hat{X}^1_t - \frac{1}{N} \left( \hat{X}_{t}^1 +\sum_{j=2, \ldots, N}X_{t}^j \right) \right) \\
&\leq& h \left( \hat{X}^1_t - \frac{1}{N} \hat{X}_{t}^1 - \frac{N-1}{N} m_{0-} \right)  + h^{\prime} \left( \hat{X}^1_t - \frac{1}{N} \hat{X}_{t}^1 - \frac{N-1}{N} m_{0-} \right)\left(\frac{\sum_{j=2}^{ N}X_{t}^j}{N} -\frac{N-1}{N}m_{0-} \right)\\
&\quad& + \frac{K}{2}\left(\frac{\sum_{j=2}^{N}X_{t}^j}{N} -\frac{N-1}{N}m_{0-} \right)^2\\
&\leq & h \left( \hat{X}^1_t - \frac{1}{N} \hat{X}_{t}^1 - \frac{N-1}{N} m_{0-} \right)  + K \left|\hat{X}^1_t - \frac{1}{N} \hat{X}_{t}^1 - \frac{N-1}{N} m_{0-} \right|\cdot \left|\frac{\sum_{j=2}^{ N}X_{t}^j}{N} -\frac{N-1}{N}m_{0-} \right|\\
&& +\frac{K}{2}\left(\frac{\sum_{j=2}^N X_{t}^j}{N} -\frac{N-1}{N}m_{0-} \right)^2.
\end{eqnarray*}
Similarly,
\begin{eqnarray*}
&& h \left( \hat{X}^1_t - \frac{1}{N} \left( \hat{X}_{t}^1 +\sum_{j=2, \ldots, N}X_{t}^j \right) \right)\\
 &\geq& h \left( \hat{X}^1_t - \frac{1}{N} \hat{X}_{t}^1 - \frac{N-1}{N} m_{0-} \right)  - K \left|\hat{X}^1_t - \frac{1}{N} \hat{X}_{t}^1 - \frac{N-1}{N} m_{0-} \right|\cdot \left|\frac{\sum_{j=2}^ N X_{t}^j}{N} -\frac{N-1}{N}m_{0-} \right|\\
&& -\frac{K}{2}\left(\frac{\sum_{j=2}^N X_{t}^j}{N} -\frac{N-1}{N}m_{0-} \right)^2.
\end{eqnarray*}
{Moreover, under the control (\ref{optimal-mfg-n-player-game}), $X_t^j$ ($j=2,3,\cdots,N$) are independent and identically distributed and $|X_t^j-m_{0-}|\leq c$ a.s.. Therefore, 
\begin{eqnarray*}
\mathbb{E}\left(\frac{\sum_{j=2}^N X_t^j}{N}-\frac{N-1}{N} m_{0-}\right)^2 = \frac{\sum_{j=1}^N Var(X_t^j)}{N^2}\leq \frac{c^2}{N} = O(\frac{1}{N}),
\end{eqnarray*}
and 
\begin{eqnarray*}
\mathbb{E}\left|\frac{\sum_{j=2}^N X_t^j}{N}-\frac{N-1}{N} m_{0-}\right|\leq \left(\mathbb{E}\left(\frac{\sum_{j=2}^N X_t^j}{N}-\frac{N-1}{N} m_{0-}\right)^2\right)^{1/2} = O(\frac{1}{\sqrt{N}}).
\end{eqnarray*}
Therefore by the boundedness of $X_t^j$ ($j=2,3,\cdots,N$) and by the Fubini Theorem,
\begin{eqnarray*}
\mathbb{E}\int_0^{\infty}e^{-\alpha t}\left(\frac{\sum_{j=2}^N X_t^j}{N}-\frac{N-1}{N} m_{0-}\right)^2  dt = \int_0^{\infty}e^{-\alpha t}\mathbb{E}\left(\frac{\sum_{j=2}^N X_t^j}{N}-\frac{N-1}{N} m_{0-}\right)^2  dt = O\left( \frac{1}{N}\right).
\end{eqnarray*}
Similarly, when $\hat{X}_t^1$ is under the threshold-type control,
\begin{eqnarray}
\mathbb{E}\int_0^{\infty}e^{-\alpha t} \left|\hat{X}^1_t - \frac{1}{N} \hat{X}_{t}^1 - \frac{N-1}{N} m_{0-} \right|\cdot \left|\frac{\sum_{j=2}^{ N}X_{t}^j}{N} -\frac{N-1}{N}m_{0-} \right|  dt = O(\frac{1}{\sqrt{N}}).
\end{eqnarray}}
Now, to minimize the following payoff function 
\begin{eqnarray}
&&\mathbb{E} \int_s^{\infty}e^{-\alpha t} \left[h \left( \hat{X}_{t}^1-\frac{1}{N}\hat{X}_{t}^1-\frac{1}{N}\sum_{j=2}^N X_{t}^j \right)dt + d \eta_t^+ + d \eta_t^- \right]  \nonumber \\
&=& \mathbb{E}\int_s^{\infty}e^{-\alpha t} \left[h\left( \frac{N-1}{N}\hat{X}^1_t  - \frac{N-1}{N} m_{0-}\right)dt +d \eta_t^+ + d \eta_t^- \right] +O \left( \frac{1}{\sqrt{N}} \right). \label{slln} 
\end{eqnarray}
 is equivalent to solving the original fuel follower problem (\ref{fuelproblem})
 with a modified  running cost  $h(\frac{N-1}{N}(\cdot-m_{0-}))$. Since the  value function for 
  (\ref{fuelproblem}) is of a linear growth, 
\begin{eqnarray}\label{new_prob}
 (\ref{slln}) &\geq & \mathbb{E}\int_s^{\infty}e^{-\alpha t} \left[ h\left( \frac{N-1}{N}(\hat{X}^1_t  - m_{0-})\right)dt + d \eta_t^{1*,+} + d  \eta_t^{1*,-} \right] +O \left( \frac{1}{\sqrt{N}} \right)\\ 
 &=&  \mathbb{E} \left[ {{v}}^*({X}^1_{0-})\right]+O \left( \frac{1}{{N}} \right)+O \left( \frac{1}{\sqrt{N}} \right)\label{new_prob_exp2}
\end{eqnarray}
where $v^*({x})$ is defined in (\ref{stationary-solution}) and the expectation in (\ref{new_prob_exp2}) is with respect to the initial distribution $\mu_{0-}$.
The above analysis holds for any $(\eta_t^+,\eta_t^-) \in  \mathcal{U}^i_N$ such that $(\pmb{\xi}^{-i*},{\eta}) \in \mathcal{S}_N$. Hence the conclusion.
\end{proof}

\section{Discussions}\label{section:discussion}

\subsection{Multiple Explicit NEs for $N=2$}
When $N=2$, $h$ is symmetric with $h(X_t^1-m_t^{(2)}) = h(X_t^2-m_t^{(2)}) = h\left(\frac{X_t^1-X_t^2}{2}\right)$. This symmetry  simplifies significantly the solution structure and allows for the construction of multiple NEs.
Indeed, given the partition $Q_i$ in (\ref{region-i-charge}) for $N=2$,
$Q_1=0$, $Q_2 = \mathbb{R}^2$, one can write
 the NE and their corresponding values explicitly.
\begin{eqnarray}\label{nash-opt-control}
\begin{aligned} 
\xi^{2*}_t &= (\xi^{2*,+}_t,\xi^{2*,-}_t)\\
& = \left( \max \left\{ 0,  \max_{0 \leq u \leq t} \{  - x^2+x^1  - B_u^2 +B_u^1 + \xi_u^{2*,-}-c_2 \}  \right\}, \right.\\
&\hspace{10pt}\left. \max \left\{ 0, \max_{0 \leq u \leq t}\{x^2-x^1 +B_u^2 - B_u^1 + \xi_u^{2*,+} - c_2\} \right\} \right),\\
\xi^{1*}_t &= \left(\xi^{1*,+}_t,\xi^{1*,-}_t \right) = (0,0),
\end{aligned}
\end{eqnarray}
where $c_2>0$ is the unique positive solution of
\begin{eqnarray}\label{c0}
\frac{1}{\sqrt{\alpha}} \tanh\left(\sqrt{\alpha}x \right)= \frac{p_2^{\prime}(x)-1}{p_2^{\prime \prime}(x)},
\end{eqnarray}
with
 \begin{eqnarray*}\label{constant_c_0}
  p_2(x) = \mathbb{E} \left[ \int_0^{\infty}e^{-\alpha t}h \left(\frac{x}{2}+\frac{\sqrt{2}B_t}{2} \right)dt \right].
  \end{eqnarray*}
And the NE values are
  \begin{eqnarray}\label{v2}
    {v}^{2}(x^1,x^2)=\left\{
                \begin{array}{ll}
          v^2(x^1,x^1-c_2) - c_2-x^2+x^1, &  x^2-x^1 \leq - c_2 ,\\
                -\frac{p_2^{\prime \prime}(c_2) \cosh\left(\sqrt{\alpha}(x^2-x^1)\right)}{\alpha  \cosh\left(c_2\sqrt{\alpha}\right)} + p_2(x^2-x^1), & \vert x^2-x^1 \vert < c_2,\\
         x^2 -x^1- c_2  +v^2(x^1,x^1+c_2) , &  x^2-x^1 \geq c_2 ,
                \end{array}
              \right. 
 \end{eqnarray}
and
\begin{eqnarray}\label{v1}
   v^{1}(x^1,x^2)=\left\{
                \begin{array}{ll}
          v^{1}(x^1,x^1+c_2), &   x^1-x^2 \leq - c_2 ,\\
                -\frac{p_2^{\prime \prime}(c_2) \cosh \left(\sqrt{\alpha}(x^1-x^2)\right)}{\alpha  \cosh \left(c_2\sqrt{\alpha}\right)} + p_2(x^1-x^2),& \vert x^2- x^1  \vert < c_2,\\
        v^{1}(x^1,x^1-c_2) , &  x^1-x^2 \geq   c_2.
                \end{array}
              \right. 
  \end{eqnarray} 

\noindent There is in fact more than one NE. 
For instance, in addition to the above constructed NE, labeled as {\bf Case 1},   there are more NEs, including

\begin{enumerate}[font=\bfseries,leftmargin=3cm]
\item[Case 2:] $ \mathcal{A}_1 = \{(x^1,x^2) \given x ^1-x^2 > c_2 \mbox{ or } x ^1-x^2 < -c_2\} $ and $ \mathcal{A}_2 = \emptyset $,
\item[Case 3:] $ \mathcal{A}_1 = \{(x^1,x^2) \given x ^1-x^2 < -c_2\}$ and $ \mathcal{A}_2 = \{(x^1,x^2) \given x ^1-x^2 > c_2\}$,
\item[Case 4:] $ \mathcal{A}_1 = \{(x^1,x^2) \given x ^1-x^2 > c_2\}$ and $ \mathcal{A}_2 = \{(x^1,x^2) \given x ^1-x^2 < -c_2\}$.
\end{enumerate}
In {\bf Case 4}, clearly
\begin{eqnarray*}
\xi_t^{1*} &=&- \max \left\{ 0, \max_{0 \leq u \leq t}\left\{0,x^1-x^2 +B_u^1 - B_u^2 - \xi_u^{2*}-c_2\right\} \right \} ,\\
\xi_t^{2*} &= &-\max \left\{ 0,  \max_{0 \leq u \leq t} \left\{ 0,  x^2-x^1 +B_u^2 - B_u^1 - \xi_u^{1*}-c_2 \right\}  \right\} ,
\end{eqnarray*}
and the associated NE values  are
  \begin{eqnarray*}
    v^{1}(x^1,x^2)=\left\{
                \begin{array}{ll}
                v^1( x^1,x^1+c_2) , &   x^1-x^2 \leq -c_2, \\
                -\frac{p_2^{''}(c_2) \cosh\left(\sqrt{\alpha}(x^1-x^2)\right)}{\alpha  \cosh \left(c_2\sqrt{\alpha}\right)} + p_2(x^1-x^2), & \vert x^1-x^2 \vert < c_2,\\
         x^1 -x^2- c_2  +v^1(x^2+c_2,x^2),    & x^1-x^2 \geq c_2,
                \end{array}
              \right.
 \end{eqnarray*}
and
\begin{eqnarray*}
    v^{2}(x^1,x^2)=\left\{
                \begin{array}{ll}
                v^2( x^2+c_2,x^2), & x^2-x^1 \leq -c_2, \\
                  -\frac{p_2^{''}(c_2) \cosh \left(\sqrt{\alpha}(x^2-x^1)\right)}{\alpha  \cosh \left(c_2\sqrt{\alpha}\right)} + p_2(x^2-x^1),&  \vert x^2-x^1 \vert < c_2,\\
          x^2 -x^1- c_2  +v^2(x^1,x^1+c_2),       & x^2-x^1 \geq c_2.
          
                \end{array}
              \right.
  \end{eqnarray*} 
 Figure  \ref{figure_multi_NE}  illustrates all four NEs.
\begin{figure}[H] 
       
      \centering \includegraphics[width=0.7\columnwidth]{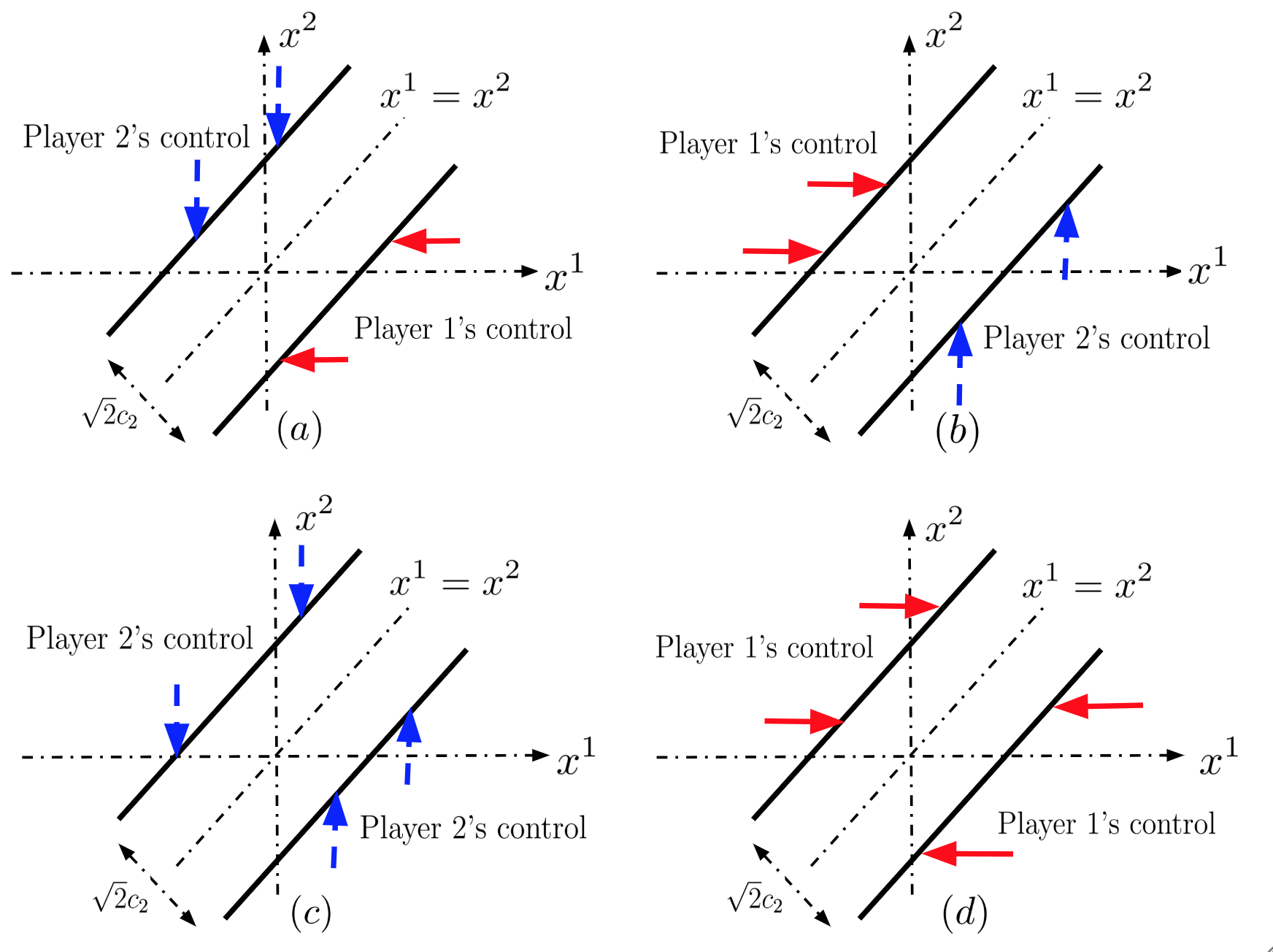}
       
        \caption{
                \label{figure_multi_NE}
        Four NEs when $N=2$
        }
\end{figure}

\subsection{With Varying $\alpha$}
\begin{proposition}\label{proposition:sensitivity}

When $h(x)=x^2$ and  $\alpha \geq 2^{-\frac{1}{3}}\frac{N-1}{N}$, $c_N$ increases with respect to $\alpha$.
\end{proposition}
The proposition follows from simple calculations.
Take $h(x)=x^2$, $\frac{p_N^{\prime}(x)-1}{p_N^{\prime \prime}(x)} = x-  \frac{\alpha k_N^2}{2}$ with $k_N = \frac{N}{N-1}>1$. Rewrite $f_N$ as $f_N(x,\alpha) = \frac{\sqrt{k_N}}{\sqrt{2\alpha}} \tanh \left( \frac{\sqrt{2\alpha}}{\sqrt{k_N}} x \right)-x + \frac{k_N^2\alpha }{2}.
$
Then
$\frac{\partial f_N}{\partial x} = -\tanh^2  \left( \frac{\sqrt{2\alpha}}{\sqrt{k_N}} x \right)$ and 
$\frac{\partial f_N}{\partial \alpha}= 
\frac{x}{2\alpha}\left(1- \tanh ^2\left( \frac{\sqrt{2\alpha}}{\sqrt{k_N}} x \right)\right)
 -\frac{\sqrt{k_N}}{2\alpha\sqrt{2\alpha}}\tanh \left( \frac{\sqrt{2\alpha}}{\sqrt{k_N}} x \right) + \frac{k_N^2 }{2}.$
One can verify that
$\frac{\partial f_N}{\partial x}<0$  for any $\alpha$ and $\frac{\partial f_N}{\partial \alpha}>0$ when $\alpha > 2^{-\frac{1}{3}}k_N^{-1}$.
Hence $\frac{\partial c_N}{\partial \alpha} >0$ when $\alpha > 2^{-\frac{1}{3}}k_N^{-1}$ follows from  the chain rule and from $f(c_N(\alpha),\alpha) =0$ for any $N$.

Figure \ref{fig:c_N}, illustrates the convergence of $c_N$ with different discount factor $\alpha$. The value of $c$ is  shown in the red dash line.


\begin{figure}[H] 
\centering
\begin{subfigure}{.3\textwidth}
  \centering
  \includegraphics[width=1.0\linewidth]{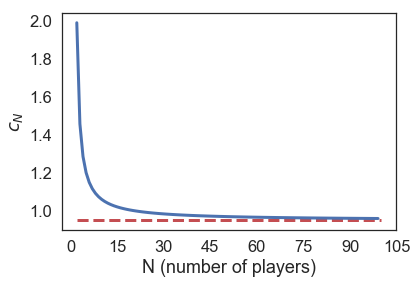}
  \caption{$\alpha=0.2$}
  \label{fig:c_N_1}
\end{subfigure}
\begin{subfigure}{.3\textwidth}
  \centering
  \includegraphics[width=1.0\linewidth]{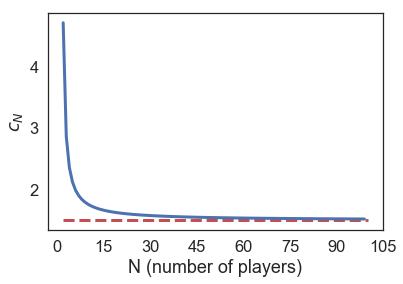}
  \caption{$\alpha=2$}
  \label{fig:c_N_2}
\end{subfigure}%
\begin{subfigure}{.3\textwidth}
  \centering
  \includegraphics[width=1.0\linewidth]{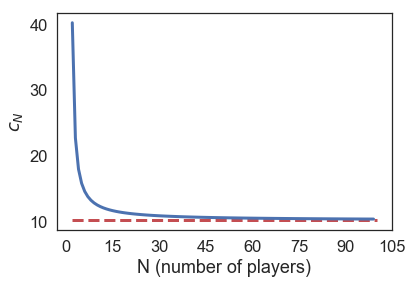}
  \caption{$\alpha=20$}
  \label{fig:c_N_3}
\end{subfigure}
\caption{\label{fig:c_N}
             Convergence of $c_N$ with different discount factors}
\end{figure}

{
\begin{remark} \label{remark1} Figure~\ref{fig:c_N} indicates that $c_N$ is a decreasing function of $N$ for any given discount factor $\alpha$. This implies that players will intervene  more frequently with  more players in the game.
Meanwhile,  $c$ being a decreasing function of $\alpha$ indicates that the bigger the $\alpha$, the less frequent players will intervene. These are consistent with  Figure~\ref{fig:v_N}.
\end{remark}
}

It is worth noting that the analysis for $\alpha_1=\cdots=\alpha_N=\alpha$ can be easily extended to the cases when $\alpha_i$'s are different.  The exact forms of the NEs, however, may be more complicated, as illustrated in the case of $N=2$ below. 

When $N=2$, denote $\alpha_i$ as the discount parameter for player $i$ ($i=1,2$).  Denote $c_2^{(i)}>0$ as the unique solution of
\begin{eqnarray}\label{c0-2}
\frac{1}{\sqrt{\alpha_i}} \tanh\left(\sqrt{\alpha_i}x \right)= \frac{p_2^{\prime}(x,\alpha_i)-1}{p_2^{\prime \prime}(x,\alpha_i)},
\end{eqnarray}
with 
 \begin{eqnarray*}\label{constant_c_0-2}
  p_2(x,\alpha_i) = \mathbb{E} \left[ \int_0^{\infty}e^{-\alpha_i t}h \left(\frac{x}{2}+\frac{\sqrt{2}B_t}{2} \right)dt \right].
  \end{eqnarray*}

\begin{corollary}[$N=2$ with $\alpha_1\ne \alpha_2$]\label{MNE-2}
Assume {\bf A1} for game (\ref{N-game}). If $\alpha _2 > \alpha_1 > 2^{-\frac{4}{3}}$, then  $c_2^{(2)} >c_2^{(1)}$. The following controls
\begin{eqnarray*}
\xi_{t}^{2*,+} &=& \textbf{1}_{\left\{x^2-x^1<-c_2^{(2)}\right\}} \left(-c_2^{(2)}-x^2+x^1 \right),\\
\xi_{t}^{2*,-} &=& \textbf{1}_{\left\{x^2-x^1>c_2^{(2)}\right\}}\left(c_2^{(2)}-x^2+x^1\right),\\
\end{eqnarray*}
and 
\begin{eqnarray*}
\xi_t^{1*,+} &=& \max \left\{0,\max_{0 \leq u \leq t} \{  x^2 + B_u^2 +\xi_0^{2*,+}-\xi_0^{2*,-}-x^1 - {B}^1_u + \xi_u^{1*,-}  - c_2^{(1)} \}\right \},\\
\xi_t^{1*,-} &=& \max \left\{0,\max_{0 \leq u \leq t} \{-x^2 - B_u^2 -\xi_0^{2*,+}+\xi_0^{2*,-}+x^1 +{B}^1_u + \xi_u^{1*,+}  - c_2^{(1)} \} \right\}.
\end{eqnarray*}
give a Markovian NE.  The corresponding NE values are 
   \begin{eqnarray}\label{v1-2}
    v^{1}(x^1,x^2)=\left\{
                \begin{array}{ll}
                v^1({x^1,x^1+c_2^{(2)}})  & x^1-x^2  \leq -c_2^{(2)},\\
                v^1( x^2-c_2^{(1)},x^2) +x^2-x^1-c_2^{(1)},  & -c_2^{(2)}  \leq   x^1-x^2 \leq -c_2^{(1)}, \\
                -\frac{p_2^{''}(c^{(1)}_2) \cosh\left(\sqrt{\alpha}(x^1-x^2)\right)}{\alpha  \cosh \left(c^{(1)}_2\sqrt{\alpha}\right)} + p_2(x^1-x^2),  & \vert x^1 -x^2\vert  \leq  c^{(1)}_2,\\
         x^1 -x^2- c^{(1)}_2  +v^1(x^2+c^{(1)}_2,x^2),    &  c^{(1)}_2  \leq x^1-x^2  \leq  c^{(2)}_2,\\
       v^1({x^1,x^1-c_2^{(2)}})   & x^1-x^2  \ge c^{(2)}_2,
                \end{array}
              \right.
 \end{eqnarray}
and
\begin{eqnarray}\label{v2-2}
    v^{2}(x^1,x^2)=\left\{
                \begin{array}{ll}
                v^2( x^1,x^1-c^{(2)}_2){+x^1-x^2-c_2^{(2)}}, & x^2- x^1 \leq-c^{(2)}_2, \\
                    {   v^2( x^2+c_2^{(1)},x^2)             }               &-c^{(2)}_2 \leq x^2- x^1 \leq-c^{(1)}_2        ,\\
                  -\frac{p_2^{''}({c^{(1)}_2}) \cosh \left(\sqrt{\alpha}(x^2-x^1)\right)}{\alpha  \cosh \left({c^{(1)}_2}\sqrt{\alpha}\right)} + p_2(x^2-x^1),& \vert x^2-x^1 \vert\leq  {c^{(1)}_2},\\
            {v^2(x^2-c_2^{(1)},x^2)}                                                   &c^{(1)}_2\leq  x^2-x^1 \leq c^{(2)}_2         ,\\
          x^2 -x^1- c^{(2)}_2  +v^2(x^1,x^1+c_2^{(2)}),       & x^2-x^1 \geq c^{(2)}_2.
          
                \end{array}
              \right.
  \end{eqnarray}

\end{corollary}
Figure \ref{figure5}  shows the NE defined in Corollary \ref{MNE-2}.
\begin{figure}[H] 
       
      \centering \includegraphics[width=0.5\columnwidth]{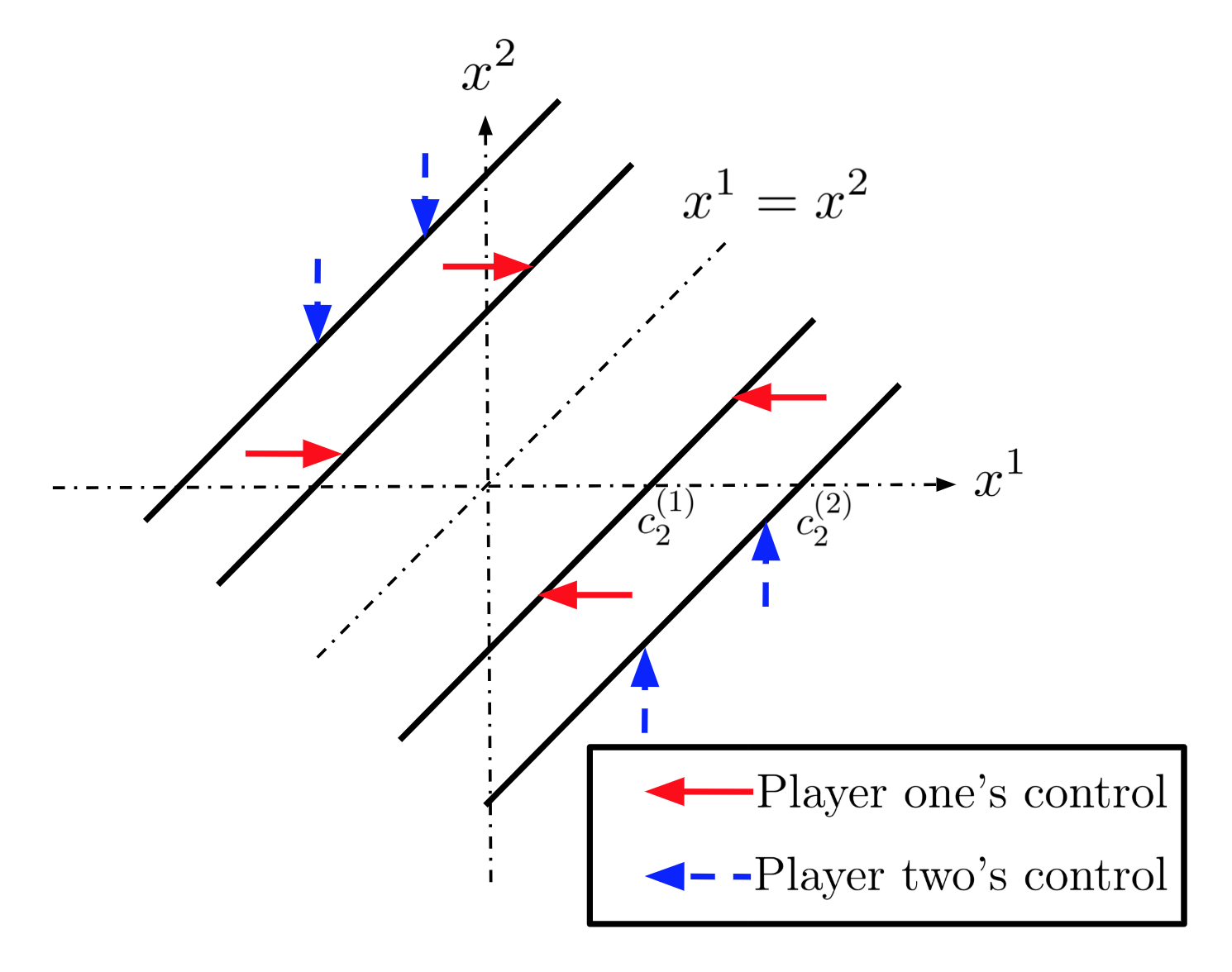}
       
        \caption{
                \label{figure5}
             N=2 with different $\alpha$ values
        }
\end{figure}



\paragraph{Acknowledgments.} The authors are in debt to the referee and the AE for their invaluable suggestions and extremely  insightful remarks. 

\section*{Appendix A: the Skorokhod Problem (SP) } 
 First, some notation for a general polyhedron G.
  
Take a fixed integer $l$ ($l \geq 1$), let $\pmb{J}=\{1,2,\cdots,l\}$.
Given an $l-$dimensional vector $\pmb{b}=(b_1,\cdots,b_l)$ and 
$N$-dimensional unit vectors $\{\pmb{n}_j, j \in \pmb{J}\}$, a polyhedron $G$ is defined by
\[
G = \{\pmb{x} \in \mathbb{R}^N \ \vert \ \pmb{n}_j \cdot \pmb{x} >b_j \mbox{ for any } j \in \pmb{J}\}.
\]
Assume the faces 
 $F_j = \{\pmb{x} \in \bar{G} \ \vert \ \pmb{n}_j \cdot \pmb{x}=b_j\}$  ($j \in \pmb{J}$)
are of  dimension ${N}-1$.

Next, take  another set of $N-$dimensional vectors $\{\pmb{d}_j, j\in \pmb{J}\}$, we can define the SP problem on a polyhedron with oblique reflections
$(\pmb{d_1}, \cdots, \pmb{d_l})$, in both the strong sense and the weak sense.

\begin{definition} [Strong solution to SP]\label{solution_skorokhod} 
Given a polyhedron $G$, a vector field  $(\pmb{d_1}, \cdots, \pmb{d_l})$, and  $\pmb{x} \in \overline{G}$.
 Given an $N$-dimensional Brownian motion $\{\pmb{B}_t\}_{t \geq 0}$ on the probability space $(\Omega, \mathcal{F}, \mathbb{P}_{\pmb{x}})$,
 a strong solution to the SP with the data   $\left(\pmb{x},   G, (\pmb{d_1}, \cdots, \pmb{d_l}), 
 \{B_t\}_{t \geq 0}\right)$ is an $\mathcal{F}_t^B$-adapted process $\pmb{X}_t$ such that

 \begin{enumerate}[label=(\alph*)]
  \item $
\pmb{X}_t =\pmb{x} +\pmb{B}_t + \pmb{\eta}_t D$, with $D =  \begin{bmatrix}
\pmb{d}^1\\
\cdots\\
\pmb{d}^l
\end{bmatrix}\in \mathbb{R}^{l\times N}$,
 \item $\pmb{X}_t$ has a  continuous path in $\overline{G}$, 
 \item  $\pmb{X}_t \in \overline{G}, \quad \mbox{ for any } t \geq 0 \mbox{ a.s.}$,
 \item $\eta^j_0=0$, $\eta^j_t$ is continuous and nondecreasing, $\eta^j_t$  increases only when $\pmb{X}_t$ is on the face $F_j$. That is, 
\[
{\eta}^j_t = \int_0^t \textbf{1}_{\left\{\pmb{X}_s \in \partial F_j\right\}} d {\eta}^j_s,
\]
\item  the reflection direction 
  $ \gamma(\pmb{x}):= \pmb{d}_j, \mbox{ if } \pmb{x} \in F_j \mbox{ for } j \in \pmb{J}.$
\end{enumerate}
\end{definition}

\begin{definition}[Weak solution to SP] \label{solution_skorokhod_weak} Given a polyhedron $G$, 
a vector field  $(\pmb{d_1}, \cdots, \pmb{d_l})$, and  $\pmb{x} \in \bar{G}$.
A weak solution to the SP with the data   $\left(\pmb{x},   G, (\pmb{d_1}, \cdots, \pmb{d_l})\right)$ is
 an adapted $N$-dimensional process $\pmb{X}_t$ defined on some probability space $(\Omega, \mathcal{F}, \mathbb{P}_{\pmb{x}})$ such that
 \begin{enumerate}[label=(\alph*)]
\item $
 \pmb{X}_t = \pmb{W}_t +  \pmb{\eta}_t D$, with $D =  \begin{bmatrix}
\pmb{d}^1\\
\cdots\\
\pmb{d}^l
\end{bmatrix}\in \mathbb{R}^{l\times N}$ and $\pmb{W}$  an $N$-dimensional Brownian motion under $\mathbb{P}_{\pmb{x}}$, with  $\pmb{X}_{0}=\pmb{x}$, $\mathbb{P}_{\pmb{x}}$-a.s.,
 \item $\pmb{X}_t$ has a continuous path in $\overline{G}$, $\mathbb{P}_{\pmb{x}}$-a.s., 
  \item $\eta^j_0=0$, $\eta^j_t$ is continuous and nondecreasing,
 $\eta^j(t)$ can increase only when $\pmb{X}_t$ is on the face $F_j$. That is,
\[
{\eta}^j_t = \int_0^t \textbf{1}_{\left\{\pmb{X}_s \in \partial F_j\right\}} d {\eta}^j_s,
\]
\item the reflection direction 
  $ \gamma(\pmb{x}):= \pmb{d}_j, \mbox{ if } \pmb{x} \in F_j \mbox{ for } j \in \pmb{J}.$
\end{enumerate}

\end{definition}

Now the proof of Theorem \ref{thm: solution skorokhod} follows from the following two lemmas.

 \begin{lemma}[Existence of the weak solution to SP]\label{lm: weak_existence_and_uniqueness}
 Fix $\pmb{x} \in \overline{\mathcal{CW}}$. There exists a weak solution to 
 the SP with the data $(\mathcal{CW},(\pmb{d}_1,\cdots,\pmb{d}_{2N}), \pmb{x})$ with $\pmb{d}_j$ $(j=1,2,\cdots,2N)$  defined in (\ref{di}) { and $\mathcal{CW}$ defined in \eqref{nonaction_region}}.
 It is a semimartingale reflected Brownian motion (SRBM)  starting from $\pmb{x}$. 
 \end{lemma}
 In fact, this weak solution is unique in a weak sense, see \cite{DW1996}. 
 \begin{proof} [Proof of Lemma \ref{lm: weak_existence_and_uniqueness}]
Following the notation in \cite{DW1996},
 define the \textit{maximal set} to characterize the points on $\partial \mathcal{CW}$ as follows. 
 Take $\textbf{J} =\{1,2,\cdots,2N\}$  the index set of the $2N$ faces of $\mathcal{CW}$. For each $\varnothing \neq \pmb{K} \subset \textbf{J}$, define $F_{\textbf{K}} = \cap _{j \in \textbf{K}}F_j$. Let $F_{\varnothing}=\mathcal{CW}$. A set $\pmb{K} \subset \pmb{J}$ is \textit{maximal} if $\textbf{K} \neq \varnothing$, $F_{\textbf{K}} \neq \varnothing$, and $F_{\textbf{K}} \neq F_{\tilde{\textbf{K}} }$ for any $\tilde{\textbf{K}} \supset {\textbf{K}} $ such that $\tilde{\textbf{K}} \neq {\textbf{K}}$. Now, it suffices to show that for each \textit{maximal} $\pmb{K} \subset \pmb{J}$,
  \begin{itemize}[font=\bfseries]
 \item[(S.a)] there is a positive linear combination $\pmb{d} = \sum_{j \in \pmb{K}} \ a_j \pmb{d}_j$ ($a_j >0, \ \forall j \in \pmb{K}$) of the $\{\pmb{d}_j, \ j \in \pmb{K}\}$ such that $\pmb{n}_j \cdot \pmb{d} >0$ for any $j \in \pmb{K}$;
 \item[(S.b)] there is a positive linear combination $\pmb{n} = \sum_{j \in \pmb{K}} c_j \pmb{n}_j$ ($c_j >0, \ \forall j \in \pmb{K}$) of the $\{\pmb{n}_j,\ j \in \pmb{K}\}$ such that $\pmb{d}_j \cdot \pmb{n} >0$ for any $j \in \pmb{K}$.
 \end{itemize}
 
Let us first show that for any maximal $\pmb{K}$, $|\pmb{K}| \leq N-1$. To see this claim, denote
 \[N_{\mbox{mat}}:=\begin{bmatrix}
  \pmb{n}_1 \\
   \pmb{n}_2 \\
    \vdots  \\
   \pmb{n}_{2N} 
\end{bmatrix}
=\frac{\sqrt{N-1}}{\sqrt{N}}
 \begin{bmatrix}
   1 & -\frac{1}{N-1} & -\frac{1}{N-1}& \dots  & -\frac{1}{N-1} \\
    -\frac{1}{N-1}& 1 &-\frac{1}{N-1} & \dots  & -\frac{1}{N-1} \\
    \vdots & \vdots & \vdots & \ddots & \vdots \\
     -\frac{1}{N-1} & -\frac{1}{N-1} & -\frac{1}{N-1} & \dots  & 1 \\
      -1 & \frac{1}{N-1} & \frac{1}{N-1} & \dots  & \frac{1}{N-1} \\
     \frac{1}{N-1} & -1 & \frac{1}{N-1} & \dots  & \frac{1}{N-1} \\
      \vdots & \vdots & \vdots & \ddots & \vdots \\
   \frac{1}{N-1} & \frac{1}{N-1} & \frac{1}{N-1} & \dots  & -1
\end{bmatrix} \in \mathbb{R}^{2N \times N}.
 \]

It follows from some calculations that $\det (N_{\mbox{mat}})=N-1$,  implying that for any $\pmb{K} \subset \pmb{J}$ with $|\pmb{K}|=N$, $\cap F_{j \in \pmb{K}}=\varnothing$. Moreover, for any maximal \pmb{K}, $|\pmb{K}| \leq N-1$.
Now checking the conditions  \textbf{(S.a)} and \textbf{(S.b)} for any maximal $\pmb{K}$ reduces to checking these conditions  for the maximal $\pmb{K}$ with $|\pmb{K}|=N-1$.

Note that for any $i=1, \cdots, N$, $F_i$ and $F_{N+i}$ are parallel faces such that $F_i \cap F_{N+i}=\varnothing$, there is no maximal $\pmb{K}$ for which both $i \in \pmb{K}$ and  $N+i \in \pmb{K}$.
Thus, take any $\pmb{K} = \{i_1,\cdots,i_{N-1}\}$, where $i_k \in \{k,N+k\}$ for $k=1,2,\cdots,N-1$.
Denote $m$ as the number of indexes in $\pmb{K}$ which is strictly smaller than $N$, then $N-1-m$ is the number of indexes in $\pmb{K}$ that are greater than $N$.

To check \textbf{(S.a)}, define $\pmb{n} = \sum_{k=1}^{N-1}\pmb{n}_{i_k}$, then for any $k \in \{1,2,\cdots,N-1\}$,
\begin{eqnarray*}
\pmb{n} \cdot \pmb{d}_{i_k} = \frac{\sqrt{N-1}}{\sqrt{N}}\left[1+\frac{1}{N-1}\left[ \pmb{1}_{(i_k\leq N)}(N-2m)+\pmb{1}_{(i_k> N)}(-N{+2m}+2)\right] \right]>0.
\end{eqnarray*}
To check \textbf{(S.b)}, define $d = \sum_{k=1}^{N-1}d_{i_k}$, then for any $k \in \{1,2,\cdots,N-1\}$,
\begin{eqnarray*}
\pmb{d} \cdot \pmb{n}_{i_k} = \frac{\sqrt{N-1}}{\sqrt{N}}\left[1+\frac{1}{N-1}\left[ \pmb{1}_{(i_k\leq N)}(N-2m)+\pmb{1}_{(i_k> N)}(-N{+2m}+2)\right] \right]>0.
\end{eqnarray*}

 \end{proof}

Next, the  uniqueness of solution in the strong sense is established by the localization technique. 
That is, construct a sequence of bounded region $\mathcal{W}_{k}$ ($k \in \mathbb{N}^+$) such that
 \[
 \mathcal{W}_{1} \subset  \mathcal{W}_{2}  \subset \cdots \subset \mathcal{CW},
 \]
 where $\mathcal{W}_{k}$ satisfies the condition in \cite{DI1993}. Then  define a sequence of stopping times associated with $\mathcal{W}_{k}$ ($k \in \mathbb{N}^+$) and extend the strong uniqueness result on  bounded regions in \cite{DI1993}.
\begin{lemma}[Uniqueness of the strong solution to SP]\label{lm: strong_uniqueness}
Given a probability space $(\Omega, \mathcal{F}, \mathbb{P})$, suppose there are two strong solutions $\pmb{X}_t^{*}$ and $\pmb{X}_t^{*\prime}$ to the SP with the data $\left(\pmb{x},\mathcal{CW}, (\pmb{d}_1,\cdots,\pmb{d}_{2N}),\{\pmb{B}_t\}_{t \ge 0}\right)$ with $\pmb{d}_i$ $(i=1\cdots,2N)$  defined in (\ref{di}). Then
\begin{eqnarray*}
\mathbb{P}_{\pmb{x}}\left( \pmb{X}_t^{*}=\pmb{X}_t^{*\prime};\quad 0 \leq t <\infty \right)=1.
\end{eqnarray*}
 \end{lemma}
 
 \begin{proof}[Proof of Lemma \ref{lm: strong_uniqueness}.]
First, the uniqueness on a bounded region.
To this end,  define the bounded region $\mathcal{W}_{k} = \mathcal{CW} \cap \left\{\pmb{x} \ \left\vert \  \left|\sum_{i=1}^N x_i \right| < k \right.\right\}$ for $k \in \mathbb{N}_{+}$. Clearly, $\mathcal{W}_{k}  \subseteq \mathcal{W}_{k+1} \subseteq \mathcal{CW} $ and $\mathcal{CW} = \cup_{k} \mathcal{W}_{k}$. Define the boundaries of $\mathcal{W}_{k} $ as 
\begin{eqnarray*}
\partial \mathcal{W}_{k} = \cup_{i=1}^{2N} F_i^{(k)} \cup F_{2N+1}^{(k)} \cup F_{2N+2}^{(k)},
\end{eqnarray*}
where $F_i^{(k)}  = F_i \cap \overline{ \mathcal{W}}_{k}$ for $i=1,\ldots, 2N$, $F_{2N+1}^{(k)} =\overline{ \mathcal{W}}_{k} \cap \{\pmb{x} \given \sum_{i=1}^N x_i = k\} $, and $F_{2N+2}^{(k)} =\overline{ \mathcal{W}}_{k} \cap \{\pmb{x} \given \sum_{i=1}^N x_i = -k\} $. Define the reflection direction $\gamma^{(l)}(\cdot)$ on $\partial \mathcal{W}_{k}$ as
\begin{eqnarray}\label{direction_bounded}
\gamma^{(k)}(\pmb{x})=\left\{\begin{array}{ll}
&\pmb{d}_{2N+1} =\pmb{n}_{2N+1}= \frac{1}{\sqrt{N}}(-1,-1,\ldots,-1), \,\, \pmb{x} \in F_{2N+1}^{(k)},\\
&\pmb{d}_{2N+2} =\pmb{n}_{2N+2}= \frac{1}{\sqrt{N}}(1,1,\ldots,1), \,\, \pmb{x} \in F_{2N+2}^{(k)},\\
&\pmb{d}_i, \, \, \pmb{x} \in F_{i}^{(k)}, \mbox{ for } i=1,2,\cdots,2N. 
\end{array}\right.
\end{eqnarray}

 For $\pmb{x} \in \partial \mathcal{W}_k$, define $I_k(\pmb{x}):= \left\{i: \pmb{x} \in F_i^{(k)}\right\}$ as the \textit{index set} of $\pmb{x}$. Following \cite{DI1993}, we will show that, for each $\pmb{x} \in \partial \mathcal{W}_{k}$, there exists  $b_i \geq 0$, $i \in I_k(\pmb{x})$, such that
 \begin{eqnarray*}
 b_i \pmb{d}_i \cdot \pmb{n}_i> \sum_{j \in I_k(\pmb{x}) \setminus \{i\}} b_j |\pmb{d}_j \cdot \pmb{n}_i|. \qquad \qquad \mbox{\textbf{(S.c)}}
 \end{eqnarray*}
 Define $b_i=1$ for any $i=1,2,\cdots,2N+2$. It is sufficient to verify \textbf{(S.c)} for $\pmb{x} \in \partial \mathcal{W}_k$ such that $|I_k(\pmb{x})|=N$. In this case, either $2N+1 \in I_k(\pmb{x})$ or $2N+2 \in I_k(\pmb{x})$. Take any $i_0 \in I_k(\pmb{x})\setminus \{2N+1,2N+2\}$, 
 \begin{eqnarray*}
 |\pmb{n}_{i_0}\cdot \pmb{d}_{i_0}| &=& \frac{\sqrt{N-1}}{\sqrt{N}},  \\
 | \pmb{n}_{i_0}\cdot \pmb{d}_{j}| &=&  \frac{\sqrt{N-1}}{\sqrt{N}}\frac{1}{N-1},\quad \mbox{ for }j \in  I_k(\pmb{x})\setminus \{2N+1,2N+2,i_0\}, \\
  | \pmb{n}_{i_0}\cdot \pmb{d}_{j}| &=& 0 ,\quad \mbox{ for }j \in  \{2N+1,2N+2\}.
 \end{eqnarray*}
 Hence (\textbf{S.c}) holds with $\pmb{d}_{i_0}\cdot \pmb{n}_{i_0}= \frac{\sqrt{N-1}}{\sqrt{N}}$ and $\sum_{j \in I_k(\pmb{x}) \setminus \{i_0\}} |\pmb{d}_j \cdot \pmb{n}_{i_0}|= \frac{\sqrt{N-1}}{\sqrt{N}} \frac{N-2}{N-1}$.
 By \cite{DI1993}, there exists a unique strong solution $(\pmb{X}^{k}_t,\pmb{\eta}^{k})_{t\geq 0}$ to the SP with the data $\left(\pmb{x}, \mathcal{W}_k,(\pmb{d}_1,\cdots,\pmb{d}_{2N+2}),\{\pmb{B}_t\}_{t \geq 0}\right)$ such that $\pmb{x} \in \overline{\mathcal{W}}_k$.
 
 %
 %
%
Now, let $(\pmb{X}_t^{k \prime},\pmb{\eta}^{k\prime})$ be the strong solution to the SP with the data $\left(\pmb{x}^{\prime}, \mathcal{W}_k,(\pmb{d}_1,\cdots,\pmb{d}_{2N+2}),\{\pmb{B}^{\prime}_t\}_{t \geq 0}\right)$. 
Then by \cite{DI1993}, there exists a constant $C_k <\infty$ such that for any $0 \leq t \leq T$,
 \begin{eqnarray}\label{estimation_bound}
 \mathbb{E} \left( \sup_{ 0 \leq s \leq t} \|\pmb{X}_s^{k}- \pmb{X}_s^{k \prime}|^2\right) \leq C_k \left\{\|\pmb{x}-\pmb{x}^{\prime}\|^2 + \int_0^t \mathbb{E}\left( \sup_{ 0\leq u \leq s}\|\pmb{B}_s- \pmb{B}_s^{\prime}\|^2\right)ds\right\}.
 \end{eqnarray}
 
To finish the proof, now suppose that there are two strong solutions $(\pmb{X}_t^{*},\pmb{\eta}^*)_{t \geq 0}$ and $(\pmb{X}_t^{*\prime},\pmb{\eta}^{*\prime})_{t \geq 0}$ to the SP with the data $\left(\pmb{x}, \mathcal{CW},(\pmb{d}_1,\cdots,\pmb{d}_{2N}),\{\pmb{B}_t\}_{t \geq 0}\right)$, with $\pmb{d}_i$ $(i=1,2,\cdots,2N)$ defined in (\ref{di}) and $\pmb{X}_0^{*}=\pmb{X}_0^{*\prime}=\pmb{x} \in \overline{\mathcal{CW}}$.
Suppose there exists $M := M(\pmb{x})$ such that $\pmb{x} \in \overline{\mathcal{W}}_k$ for $k \geq M$. Define $\tau_{k} = \inf \{t: \pmb{X}_t^{*} \in F^{(k)}_{2N+1} \cup F^{(k)}_{2N+2}\}$ and $\tau_{k}^{\prime} = \inf \{t: \pmb{X}_t^{*\prime} \in F^{(k)}_{2N+1} \cup F^{(k)}_{2N+2}\}$.
Then  the uniqueness of the strong solution to SP with the data $(\pmb{x}, \mathcal{W}_k,\gamma^{(k)},\{\pmb{B}_t\}_{t \geq 0})$ implies that for $k \geq M$,
\begin{eqnarray}
&& \mathbb{P}_{\pmb{x}}\left(\pmb{X}_t^{*}= \pmb{X}_t^{*\prime}, \ t \leq \tau_k \right)=1,\\
&& \mathbb{P}_{\pmb{x}}\left(\tau_k =\tau_k^{\prime} \right)=1. \nonumber
 \end{eqnarray}
 By the continuity of the probability measure, 
 \begin{eqnarray}
\mathbb{P}_{\pmb{x}}\left(\pmb{X}_t^{*}=\pmb{X}_t^{*}{'}, \ t \leq \tau_k, k \rightarrow \infty \right)= \lim_{k \rightarrow \infty} \mathbb{P}_{\pmb{x}}\left(\pmb{X}_t^{*}= \pmb{X}_t^{*\prime}, \ t \leq \tau_k \right)=1.
 \end{eqnarray}
 Now it remains  to show $\lim_{k \rightarrow \infty} \tau_{k} = \infty$ a.s.. Suppose otherwise, then there exists $\tau^* = \tau^*(\omega)<\infty$
 such that $\lim_{k \rightarrow \infty} \tau_{k} = \tau^*$ pathwise. 
 Therefore,
 \begin{eqnarray*}\label{unbounded_stopping}
 \mathbb{P}_{\pmb{x}}\left(\ \left\vert\sum_{i=1}^N {X}_{\tau^*}^{i*}\right \vert=\infty\  \right)=\mathbb{P}_{\pmb{x}}\left(\ \left\vert\sum_{i=1}^N {x}^i+{B}_{\tau^*}^i+\eta_{\tau^*}^{i*}-\eta_{\tau^*}^{N+i*}\right \vert=\infty\  \right)=1,
 \end{eqnarray*}
 which implies, from the bounded variation property of  $\{\pmb{\eta}^*\}_{t \ge 0}$,
 $ \mathbb{P}_{\pmb{x}}\left(\left\vert\sum_{i=1}^N B_{\tau^*}^i\right\vert = \infty \right)=1.$
 This contradicts with the property of Brownian motion, thus $\lim_{k \rightarrow \infty} \tau_{k} = \infty$ a.s..
  \end{proof}

\section*{Appendix B: Well-posedness of Algorithm \textbf{1}}

If $\pmb{x}=(x^1, \cdots, x^N)\notin \overline{
\mathcal{CW}}$, then there exists an $i$ such that $\pmb{x} \in  \mathcal{A}_i$.
For any $k>1$, denote the point after the $k$-th jump as $\pmb{x}_{k} = (x^1_{k},\ldots,x^N_{k})$.
In step $k+1$, if $\pmb{x}_{k} \in \mathcal{A}_i$, player $i$ will apply a minimal push to reach the boundary $\partial E_i^{-} \cup \partial E_i^{+}$.

If the jumps do not stop in finite steps, an argument by contradiction will show that  
they converge to $\hat{\pmb{x}} \in \partial {\mathcal{CW}}$.
Let us first  show that  $\{\pmb{x}_{k}\}_{k\ge 1}$ converges. At each step $k \ge 1$, denote $x^{(1)}_{k} \leq \cdots \leq x^{(N)}_{k} $ as the ordered points of $x^1_{k}, \ldots,x^N_{k}$. At each step $k$, only  the player with position ${x}^{(1)}_{k}$ or ${x}^{(N)}_{k}$ will jump.
Therefore $\{x^{(1)}_{k}\}_{k\ge 1}$ is a non-decreasing sequence with an upper bound $\max_{i\le N}x^{i}$. Hence the limit exists, denoted as $x^{(1)}_*$. Similarly,  the bounded non-increasing sequence $x^{(N)}_{k}$ has a limit, denoted as $x_*^{(N)}$. Then by the sandwich argument, $\{\pmb{x}_{k}\}_{k\ge 1}$ converges.

Next, denote the distance $d^{i}_{k} = |x^{i}_{k} - \frac{\sum_{i \neq j}{x^{j}_{k}}}{N-1}|$. By definition of $ \mathcal{A}_i$, the player with the biggest $d^{i}_{k}$ will jump in step $k+1$. 
Suppose $\hat{\pmb{x}} = \lim_{k \rightarrow \infty }\pmb{x}_{k} \notin \partial {\mathcal{CW}}$, then there exists an $m \in \{1,\ldots,N\}$ such that $\hat{\pmb{x}} \in  \mathcal{A}_m$. Denote the distance $\Delta = |\hat{{x}}^m -\frac{ \sum_{j \neq m}\hat{{x}}^j}{N-1} |-c_N=\max_{i=1,2,\cdots,N} \{|\hat{{x}}^i -\frac{ \sum_{j \neq i}\hat{{x}}^j}{N-1} |\} - c_N >0$. Given $\epsilon>0 $ so that $\epsilon < \frac{\Delta}{8N}$ and $\overline{\mathcal{CW}} \cap {B}_{\epsilon}(\hat{\pmb{x}}) = \emptyset$, 
there exists a sufficiently large $K>0$ such that for any $k'>K$, $
\pmb{x}_{k{'}} \in B_{\epsilon}(\hat{\pmb{x}}).$
That is, $\sum_{i=1}^N |{x}^i_{k{'} }- \hat{{x}}^i |^2 \leq \epsilon^2$.
By the triangle inequality,
\begin{eqnarray*}
\left|x^m_{k'} - \frac{\sum_{j \neq m} x_{k'}^j}{N-1}\right|-c_N &\geq& \left|\hat{{x}}^m- \frac{\sum_{j \neq m} \hat{{x}}^j}{N-1}\right|-c_N - \left| \left( x^m_{k'}  - \frac{\sum_{j \neq m} x_{k'}^j}{N-1} \right)- \left( \hat{{x}}^m  - \frac{\sum_{j \neq m} \hat{{x}}^j}{N-1} \right) \right| \\
& \geq & \left|\hat{{x}}^m- \frac{\sum_{j \neq m} \hat{{x}}^j}{N-1} \right| -c_N- \left|x^m_{k'} -\hat{{x}}^m \right| -\frac{1}{N-1} \sum_{j \neq m} \left|x_{k'}^j - \hat{{x}}^j \right|\\
&\geq& \Delta -2 \epsilon \geq \frac{4N-1}{4N} \Delta >2 \epsilon .
\end{eqnarray*}
Thus in step $k'+1$, the player should jump at a minimum distance of $\frac{4N-1}{4N} \Delta$, which is strictly greater than $2\epsilon$ when $N>1$. Therefore $\pmb{x}_{k'+1} \notin B_{\epsilon}(\hat{\pmb{x}})$, which is a contradiction. Hence $\hat{\pmb{x}} = \lim_{k \rightarrow \infty }\pmb{x}_{k} \in \partial \overline{\mathcal{CW}}.$

To see that the total distance of sequential jumps  is bounded, rewrite  $d^{i}_{k}$ in the form of $d^{i}_{k} =\frac{N-1}{N} \left|x^{i}_{k} - \bar{x}_{k} \right|$, where $\bar{x}_{k} = \frac{\sum_{j=1}^N x^j_{k}}{N}$. Clearly, in step $k+1$, either the player with value $x^{(1)}_{k}$ or the player with value $x^{(N)}_{k}$ will jump.
By the monotonicity property of $\{x^{(N)}_{k}\}_k$ and  $\{x^{(1)}_{k}\}_k$, the total distance of jumps 
is bounded pointwise.

\begin{figure}[H] 
       \centering \includegraphics[width=0.55\columnwidth]{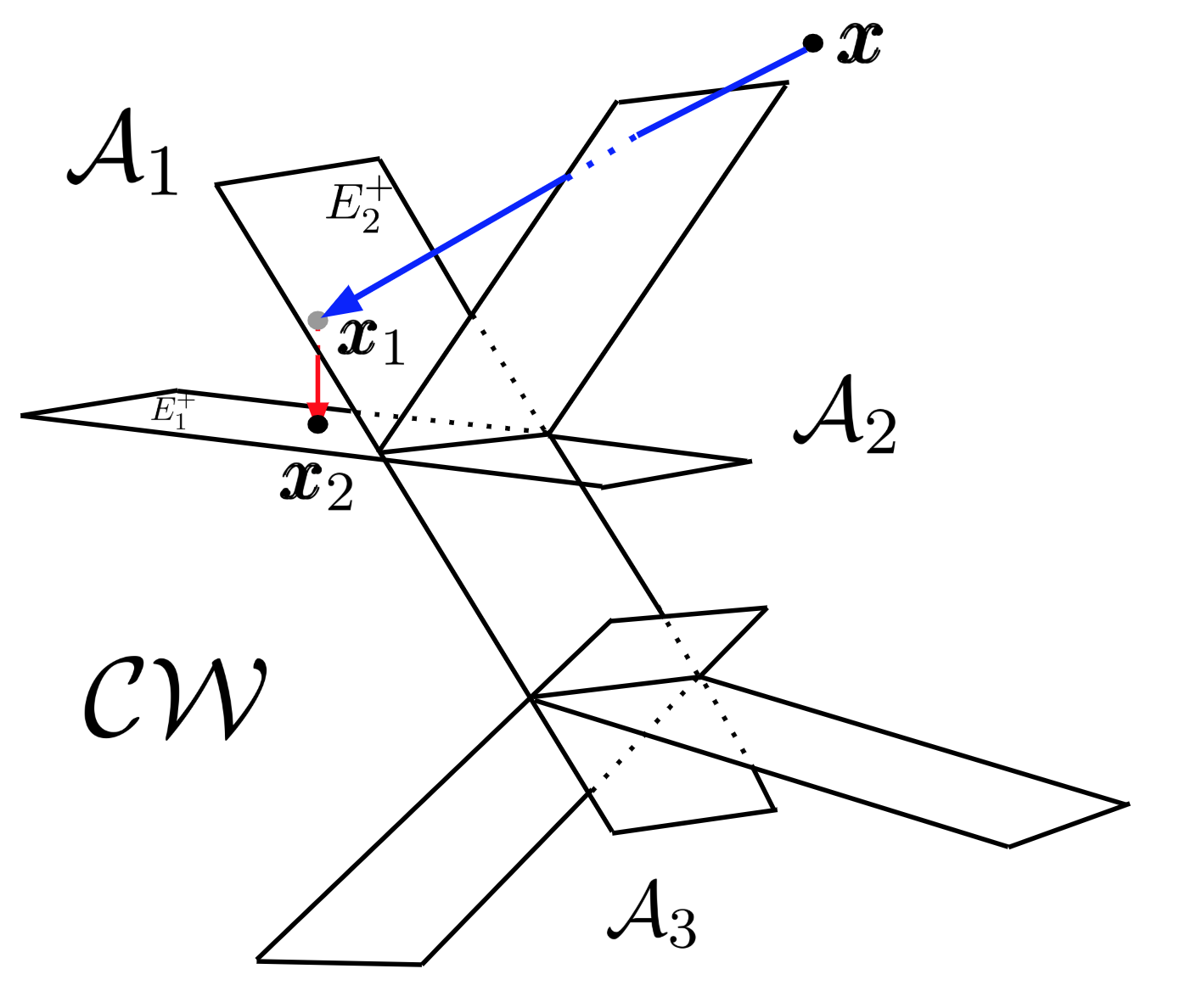}
       
        \caption{
                \label{figure_jumps} 
              Sequential jumps at time $0$
        }
\end{figure}

\section*{Appendix C: Proof of Proposition \ref{constant_convergence}}
\begin{proof}
First, denote for $N \geq 2$,
\begin{eqnarray*}\label{f_N}
f_N(x) =  \frac{1}{\sqrt{\frac{2(N-1)\alpha}{N}}} \tanh \left(x \sqrt{\frac{2(N-1)\alpha}{N}} \right) - \frac{p'_N(x)-1}{p''_N(x)},
\end{eqnarray*}
and $$f_1(x) = \frac{1}{\sqrt{{2\alpha}}} \tanh \left(x \sqrt{2\alpha} \right) - \frac{p^{\prime}_1(x)-1}{p^{\prime \prime}_1(x)}.$$
Then there exists a unique $c>0$ such that $f_1(c) = 0$ and there exists a unique  $c_N>0$ such that $f_N(c_N) = 0$ for $N \geq 2$.
Denote $m_1(x) = \frac{p^{\prime }_1(x)-1}{p^{\prime \prime}_1(x)}$ and $m_N(x) = \frac{p'_N(x)-1}{p''_N(x)}$. There exists $\tilde{c}_N>0$ such that
$m'_N(x) \geq 1$ on $(\tilde{c}_N,\infty)$ with $0<\tilde{c}_N<c_N<\infty$ for $N \geq 2$. And there exists $\tilde{c}>0$ such that
$m'_1(x) \geq 1$ on $(\tilde{c},\infty)$ with $0<\tilde{c}<c<\infty$.
 Now $0<\tanh'(x) = 1-\tanh^2(x)<1$ for any $x \in (0,\infty)$, therefore $f'_N(x)<0$ on $(c_N,\infty)$ for $N\ge 2$ and $f'_1(x)<0$ on $(c,\infty)$.
Since $f_N$ converges to $f_1$ pointwise, for any $\epsilon>0$, there exists an $N_{\epsilon}$ such that  for any $n \geq N_{\epsilon}$,
$
|f_n(c)| = |f_n(c) - f_1(c)| \leq \epsilon$.
By the uniqueness of the zeros for each function $f_N$, 
$
c_N \rightarrow c$
as $N \rightarrow \infty$.

Secondly, when $h=x^2$, $f_N$ reduces to
\[
f_N(x) = \frac{1}{\sqrt{\frac{2(N-1)\alpha}{N}}}\tanh \left(\sqrt{\frac{2(N-1)\alpha}{N}}x \right)-x+\frac{\alpha}{2\left( \frac{N-1}{N}\right)^2},
\]
with $f_N(c_N)=0$. 
Therefore, $\frac{\partial c_N }{\partial N} = - \frac{\partial f_N}{\partial N} \cdot \frac{1}{\frac{\partial f_N}{\partial c_N}}$ with $\frac{\partial f_N}{\partial c_N} = -\tanh^2\left(\sqrt{\frac{2(N-1)\alpha}{N}}x \right)<0,$ the conclusion follows after simple computations.

\end{proof}

\section*{Appendix D: Stationary Mean Field Games (SMFGs) }
{
\paragraph{SMFG for (\ref{game}).} An SMFG version of (\ref{game})  can be formulated as the follows.
\begin{eqnarray}\label{stationary_game}
\begin{aligned}
v_{\infty}(x) =& \inf_{(\xi^{+},\xi^{-}) \in  \mathcal{U}_{\infty}} J_{(\infty)} (x;\xi_t^{+},\xi_t^{-} ) \\
=& \inf_{(\xi^{+},\xi^{-}) \in  \mathcal{U}_{\infty}} \mathbb{E} \int_0^{\infty} e^{-\alpha t} \left[\left. {h}(X_t-m_{\infty})dt + d {\xi_t^{i,+}}+d {\xi_t^{i,-}} \right\vert X_{0-} = x\right],\\
 \mbox{such that} & \qquad X_t = B_t + \xi_t^ {+} -  \xi_t^{-}+x, \\
 &\hspace{20pt}\mu_{0-} = \mu, \ \ X_{0-}\sim \mu, \ \  m_{0-} = m = \int x \mu_{0-}(dx),
\end{aligned}
\end{eqnarray}
where 
\begin{itemize}
\item$\mu_t = \lim_{N \rightarrow \infty} \frac{ \sum_{i=1}^N \delta_{X_t^i }}{N}$ is the distribution of $X_t$, 
\item $m_t =  \lim_{N \rightarrow \infty} \frac{ \sum_{i=1}^N X_t^i}{N}$ is the mean position of the population at time $t$, 
\item $m_\infty=\lim_{t \rightarrow \infty}m_t$ is the limiting mean position if it exists.
\end{itemize}
The admissible control set for SMFG is $\mathcal{U}_{\infty}$ as defined in Section \ref{section: MFG}.
SMFG is a game with the long-term mean-field aggregation.
\begin{definition}[NE to  SMFG (\ref{stationary_game})]
An NE to the SMFG (\ref{stationary_game}) is a  pair of Markovian control $(\xi_t^{*,+},\xi_t^{*,-})_{t \geq 0}$ and a limiting mean position $m^*$ such that
\begin{itemize}
\item ${v}^*(x) =J_{(\infty)}\left(x;\xi^{*,+},\xi^{*,-}|\,m^*\right)=\min_{\xi \in \mathcal{U}_{\infty}}J_{(\infty)}\left(x;\xi^{+},\xi^{-}|m^*\right)$,
\item $m^*=\lim_{t \rightarrow \infty} \mathbb{E}[X_{t}^{*}]$ where $X_t^{*}$ is the controlled dynamic under $(\xi_t^{*,+},\xi_t^{*,-})_{t \ge 0}$.
\end{itemize}
${v}^*(x)$ is called the NE value of the SMFG associated with ${\xi}^*$.
\end{definition}}

\bibliographystyle{apa}
\bibliography{ref.bib}

\end{document}